\documentclass{amsart}
\usepackage{amsmath}
\usepackage{amsfonts}
\usepackage{graphicx}
\usepackage{latexsym, amscd, euscript}
\usepackage{color}
\usepackage{tikz}
\usepackage{enumitem}
\usepackage{mathrsfs}
\usepackage{hyperref}

\usepackage{amssymb}
\usepackage{epsfig}

\usepackage[utf8]{inputenc}
\usepackage[english]{babel}
 
\usepackage{amsthm}

\begin{document}

\newtheorem*{rep@theorem}{\rep@title}
\newcommand{\newreptheorem}[2]{%
\newenvironment{rep#1}[1]{%
 \def\rep@title{#2 \ref{##1}}%
 \begin{rep@theorem}}%
 {\end{rep@theorem}}}
\makeatother

\theoremstyle{plain}
 \newtheorem{lemma}{Lemma}[section]
\newtheorem{theorem}[lemma]
{Theorem }
\newreptheorem{theorem}{Theorem}
\newtheorem{corollary}[lemma]
{Corollary}
\newtheorem{proposition}[lemma]{Proposition}
\newtheorem*{thma}{Theorem 1.1'}

\theoremstyle{definition}
\newtheorem{definition}[lemma]{Definition }
\newtheorem{rmkdef}[lemma]{Remark-Definition}
\newtheorem{conjecture}[lemma]
{Conjecture }
\newtheorem{example}[lemma]
{Example }
\newtheorem{fact}[lemma]{Fact}
\newtheorem{remark}[lemma]{Remark}

\newcommand{\C}{{\mathbb C}}
\newcommand{\e}{\varepsilon}
\renewcommand{\l}{\lambda}
\newcommand{\T}{\mathcal T}
\newcommand{\In}{\operatorname{In}}
\newcommand{\Supp}{\operatorname{Supp}}
\newcommand{\s}{\sigma}
\renewcommand{\t}{\tau}
\newcommand{\V}{\mathcal V}
\newcommand{\NP}{\operatorname{NP}}
\renewcommand{\o}{\omega}
\newcommand{\Sc}{\mathscr{S}^\K_\preceq}
\newcommand{\Ker}{\text{Ker}}
\newcommand{\g}{\gamma}
\renewcommand{\O}{\mathcal{O}}
\renewcommand{\P}{\mathbb{P}}
\newcommand{\PP}{\mathcal P}
\renewcommand{\a}{\alpha}
\newcommand{\ovl}{\overline}
\newcommand{\G}{\mathcal{G}_\preceq}
\renewcommand{\(}{(\!(}
\renewcommand{\)}{)\!)}
\renewcommand{\b}{\beta}
\renewcommand{\int}{\operatorname{Int}}
\newcommand{\Int}{\operatorname{Int_{rel}}}
\newcommand{\ord}{\operatorname{ord}}
\newcommand{\K}{{\mathbb K}}
\newcommand{\Q}{{\mathbb Q}}
\newcommand{\R}{{\mathbb R}}
\newcommand{\Z}{{\mathbb Z}}
\newcommand{\sA}{{\mathcal A}}

\setlength\parindent{0pt}

\title[Support of  algebraic Laurent series]{Support of  Laurent series algebraic over the field of formal power series}

\author{Fuensanta Aroca}
\email{fuen@im.unam.mx}
\address{Instituto de Matem\'aticas, Universidad Nacional Aut\'onoma de M\'exico (UNAM), Mexico}

\author{Guillaume Rond}
\email{guillaume.rond@univ-amu.fr}
\address{Aix-Marseille Universit\'e, CNRS, Centrale Marseille, I2M, UMR 7373, 13453 Marseille, France}

\begin{abstract}
This work is devoted to the study of the support of a Laurent series in several variables which is algebraic over the ring of power series over a characteristic zero field. Our first result is the existence of a kind of maximal dual cone of the support of such a Laurent series. As an application of this result we provide a gap theorem for Laurent series which are algebraic over the field of formal power series. We also relate these results to  diophantine properties of the fields of Laurent series.\end{abstract}

\keywords{power series rings, support of a Laurent series, algebraic closure, non archimedean diophantine approximation}

\thanks{This work has been partially supported by ECOS Project M14M03. 
G. Rond was partially supported by ANR projects STAAVF (ANR-2011 BS01 009) and SUSI (ANR-12-JS01-0002-01)
F. Aroca was partially supported by PAPIIT IN108216, CONACYT 164447 and LAISLA}

\maketitle

\tableofcontents

\section{Introduction}
From the perspective of determining the algebraic closure of the field of power series in several variables this paper investigates conditions for a power series with support in a strongly convex cone to be algebraic over the ring of power series. We also develop the analogy between the classical Diophantine approximation theory and its counterpart for power series fields in several variables. Let us explain in more details the problem.\\

Let us denote by $\K(\!(x)\!)$ the field of fractions of the ring of formal power series in $n$ variables $x=(x_1,\ldots,x_n)$ with coefficients in a characteristic zero field $\K$. For simplicity we  assume that $\K$ is algebraically closed. This field $\K\(x\)$ is not algebraically closed. When $n=1$ it is well known that the algebraic closure of $\K\(x\)$ is the field of Puiseux series $\bigcup_{k\in\Z_{> 0}}\K\(x^{1/k}\)$. \\

When $n\geq 2$ there are several descriptions of algebraically closed fields containing $\K\(x\)$ \cite{McDonald:1995, Gonzalez:2000, ArocaIlardi:2009, SotoVicente}. The elements of these fields are Puiseux series whose support is included in a translated strongly convex rational cone containing ${\R_{\geq 0}}^n$. Here a rational cone means a polyhedral cone of $\R^n$ whose vertices are generated by integer coefficients vectors. More precisely, one of these descriptions is the following one:  for any given vector $\o\in{\R_{>0}}^n$ with $\Q$-linearly independent coordinates and  for every polynomial $P$ with coefficients in $\K\(x\)$  there exist a strongly convex cone $\s$ containing ${\R_{\geq 0}}^n$, such that $u\cdot\o> 0$ for every $u\in\s\backslash\{0\}$, a vector $\g\in\R^n$ and a Laurent Puiseux  series $\xi$ which is a root of $P$   such that
$$\Supp(\xi)\subset \g+\s.$$
Let us recall that for a Laurent power series $\xi=\sum_\a \xi_\a x^\a$ we define the support of $\xi$ as:
$$\Supp(\xi):=\{\a\in\Q^n\mid \xi_\a\neq 0\},$$
and a Laurent Puiseux series is a series $\xi$ whose support is in $\frac{1}{k}\Z^n$ for some integer $k\in\Z_{> 0}$.\\
For instance if $P(T)=T^2-(x_1+x_2)$ then the roots of $P$ are
$$\pm x_1^{\frac{1}{2}}\left(1+a_1\frac{x_2}{x_1}+a_2\left(\frac{x_2}{x_1}\right)^2+\cdots+a_k\left(\frac{x_2}{x_1}\right)^k+\cdots\right) \text{ if } \o_1<\o_2$$
and have support in the cone generated by $(0,1)$, $(1,0)$ and $(-1,1)$, or
$$\pm x_2^{\frac{1}{2}}\left(1+a_1\frac{x_1}{x_2}+a_2\left(\frac{x_1}{x_2}\right)^2+\cdots+a_k\left(\frac{x_1}{x_2}\right)^k+\cdots\right) \text{ if } \o_1>\o_2$$
and have support in the cone generated by $(0,1)$, $(1,0)$ and $(1,-1)$, where the $a_k$ are the coefficients of the following Taylor expansion:
$$\sqrt{1+U}=1+a_1U+a_2U+\cdots+a_kU^k+\cdots$$

But unlike the case $n=1$ these latter fields of Puiseux series (each of them depending on a given vector $\o$)  are strictly larger than the algebraic closure of $\K\(x\)$. So a natural question is to find  conditions for a Laurent Puiseux  series with coefficients in a strongly convex cone containing ${\R_{\geq 0}}^n$ to be algebraic over $\K\(x\)$.\\
Let us remark that  it is straightforward to see that a Laurent Puiseux series $\xi=\sum_{\a\in\Z^n}\xi_\a x^{\a/k}$ is algebraic over $\K\(x\)$ if and only if the Laurent  series $\tilde\xi=\sum_{\a\in\Z^n}\xi_\a x^{\a}$ is algebraic over $\K\(x\)$. Indeed if $P(x,T)$ is a nonzero vanishing polynomial of $\xi$ then $P(x^k,T)$ is nonzero vanishing polynomial of $\tilde\xi$, and if $\tilde\xi$ is algebraic over $\K\(x\)$ then $\xi$ is algebraic over $\K\(x^{1/k}\)$ which is a finite extension of $\K\(x\)$. So we can  restrict the question to the problem of algebraicity of a Laurent series with support in a strongly convex cone. The aim of this work is to provide necessary conditions for such Laurent series to be algebraic over $\K\(x\)$.\\

The conditions we are investigating are defined in terms of  the support of the given Laurent series. \\
Let us mention that the problem of determining the support of a series algebraic over $\K[x]$ or $\K[[x]]$ is an important problem related to several fields as tropical geometry (cf. for instance \cite{EKL} where the support of a rational power series is studied) or  combinatorics (cf. \cite{HM} for instance) and number theory (cf. for instance \cite{AB} for a characterization of the support of a power series algebraic over $\K[x]$ where $\K$ is a  field of positive characteristic in terms of $p$-automata, while it is still an open problem to prove that the set of vanishing coefficients of a univariate algebraic power series over a characteristic zero field is a periodic set).\\
On the other hand it is probably not possible to characterize completely the Laurent series which are algebraic over $\K[[x]]$ just in term of their support. A complete characterization of the algebraicity of Laurent series would probably involve conditions on the coefficients as it is the case for univariate algebraic power series in positive characteristic (see \cite{Ke}).\\

Our first main result, that will be very useful in the sequel, is a general construction of algebraically closed fields containing the field $\K(\!(x)\!)$. In particular it generalizes and unifies the previous constructions given in \cite{McDonald:1995, Gonzalez:2000, ArocaIlardi:2009, SotoVicente}. This result is the following one (see Section \ref{preorders} for the definition of a continuous positive order - but essentially this is a total order on $\R^n$ compatible with the addition and such that the elements of ${\R_{\geq 0}}^n$ are non-negative):

\begin{reptheorem}{TeoremaAlgCerrado}
Let $\K$ be a characteristic zero field and let $\preceq$ be a continuous positive order on $\R^n$. Then the set, denoted by $\mathscr{S}^\K_\preceq$, of  series $\xi$ for which there exist $k\in \Z_{> 0}$, $\gamma\in\Z^n$ and a rational cone $\sigma$ whose elements are non-negative for $\preceq$  and such that
$$		\Supp  (\xi )\subset (\gamma +\sigma )\cap \frac{1}{k}\Z^n $$
is an algebraically closed field containing $\K(\!(x)\!)$.
\end{reptheorem}

Let us mention that the  proof of this theorem is a direct consequence of a very nice result of F. J. Rayner \cite{Rayner:1974} that has been proven twenty years before the works \cite{McDonald:1995, Gonzalez:2000, ArocaIlardi:2009, SotoVicente}. \\

Our second result, and the most difficult one,  concerning the support conditions we were discussing before, can be summarized as follows: 
\begin{reptheorem}{main_thm}
Let $\xi$ be a Laurent power series which is algebraic over $\K(\!(x)\!)$ and which is not in the localization $\K[[x]]_{x_1\cdots x_n}$. Then there exists a  hyperplane $H\subset \R^n$  such that 
\begin{enumerate}
\item[i)] $\Supp(\xi)\cap H$ is infinite,
\item[ii)] one of the half-spaces delimitated by $H$ contains only a finite number of elements of $\Supp(\xi)$.
\end{enumerate}
\end{reptheorem}
 In fact Theorem \ref{main_thm} is more precise (see the complete statement in the core of the paper), but technical, and asserts the existence of a kind of maximal dual cone of the support of $\xi$. Its proof is essentially based on the identification of the elements of the algebraic closure of $\K(\!(x)\!)$ in the fields $\mathscr{S}^\K_\preceq$ when $\preceq$ runs over all the continuous positive orders on $\R^n$.
 This is the main tool to obtain our last main result which is the following one:

\begin{reptheorem}{gap_cor}[Gap Theorem]
Let $\xi$ be a Laurent series with support in $\g+\s$ where $\g\in\Z^n$ and $\s$ is a strongly convex cone containing the first orthant such that $\xi $ does not belong to the localization 
$ \K[[x]]_{x_1\cdots x_n}$. Let us assume that $\xi$ is algebraic over $\K[[x]]$. Let $\o=(\o_1,\ldots,\o_n)$ be in the interior of the dual of $\s$. We expand $\xi$ as
$$\xi=\sum_{i\in \Z_{\geq 0}}\xi_{k(i)}$$
where  
\begin{itemize}
\item[i)] for every $k(i) \in \Gamma=\Z\o_1+\cdots+\Z\o_n$, $\xi_{k(i)}$ is a (finite) sum of monomials of the form $cx^\a$ with $\o\cdot\a=k(i)$,
\item[ii)] the sequence $k(i)$ is a strictly increasing sequence of elements of $\Gamma$,
\item[iii)] for every integer $i$, $\xi_{k(i)}\neq 0$.
\end{itemize}

 Then there exists a constant $C>0$ such that
 $$k(i+1)\leq k(i)+C\ \  \ \ \forall i\in\Z_{\geq 0}.$$
\end{reptheorem}

This statement is similar to the following well known fact (see \cite{Fa} or \cite{Du} for a modern presentation of this): let $f$ be a  formal power series  algebraic over $\K[x]$ where $\K$ is a characteristic zero field. For every integer $k$ let $f_k$ denote the homogeneous part  of degree $k$ in the Taylor expansion of $f$. We can number these nonzero homogeneous parts by writing
$$f=\sum_{i\in\Z_{\geq 0}} f_{k(i)}$$
where $f_{k(i)}$ is the homogeneous part of degree $k(i)$ of $f$, $(k(i))_{i\in\Z_{\geq 0}}$ is strictly increasing and $f_{k(i)}\neq 0$ for every $i$. Then there exists an integer $C>0$ such that
$$k(i+1)\leq k(i)+C \ \  \ \ \forall i\in\Z_{\geq 0}.$$
This  comes from the fact that a power series algebraic over $\K[x]$ is $D$-finite when $\K$ is of characteristic zero. In some sense the proof of Theorem \ref{gap_cor} consists to reduce Theorem \ref{gap_cor} to this fact by using Theorem \ref{main_thm}.\\
\\
The paper is organized as follows. The first two sections are devoted to give basic definitions and results concerning cones and preorders on $\R^n$. In  Section \ref{alg_closed} we construct a family of algebraically closed fields containing $\K(\!(x)\!)$ (see Theorem \ref{TeoremaAlgCerrado}), each of them depending on a total order on $\R^n$. Then in Section \ref{dual_cone}, for a given Laurent series $\xi$ algebraic over $\K\(x\)$, we introduce two subsets of ${\R_{>0}}^n$, $\t_0(\xi)$ and $\t_1(\xi)$, whose definitions involve the preceding algebraically closed fields and we prove that $\t_0(\xi)$ plays the role of a maximal dual cone of $\Supp(\xi)$ (see Theorem \ref{main_thm}). 
Section \ref{gap_thmm} is devoted to the proof  of Theorem \ref{gap_cor}, which is based on Theorem \ref{main_thm} and $D$-finite power series. Finally in the last part we express some of the results in term of diophantine approximation properties for the fields of Laurent power series (see Theorem \ref{diop_cor}).\\

\textbf{Acknowledgment} We would like to thank Mark Spivakovsky for providing us the proof of Lemma \ref{LosEnterosInterseccionUnConoEstaBienOrenado}. We also thank Hussein Mourtada and Bernard Teissier for their helpful remarks. We also want to thank the referee for their useful and suitable comments and remarks.

\section{Polyhedral cones}

In this section we introduce some basic concepts of convex geometry. These concepts may be found in several books (see for example \cite{Fulton}).

A \textbf{(polyhedral) cone}   is a set of the form
$$\sigma =\langle u^{(1)}, \ldots, u^{(k)} \rangle :=\{\lambda_1 u^{(1)}+\cdots+ \lambda_k u^{(k)}; \;\lambda_i\in \R_{\geq 0},\; i=1, \ldots, k\}\subset \R^n$$
for some vectors $u^{(1)}, \ldots, u^{(k)} \in \R^n$.  
The $u^{(i)}$'s are called the \textbf{generators} of the polyhedral cone. A polyhedral cone is said to be  \textbf{rational} when it has a set of generators in $\Z^n$. 
A cone $\sigma\subset \R^n$ is rational if and only if $\sigma\cap\Z^n$ is a finitely generated semigroup.

We will denote by $e^{(1)},\ldots , e^{(n)}$ the vectors of the canonical basis of $\R^n$. With this notation the {\bf first orthant} is the polyhedral cone ${\R_{\geq 0}}^n =\langle e^{(1)},\ldots , e^{(n)}\rangle $. 

A subset $\s$ of $\R^n$ is a  \textbf{cone} if for every $s\in\s$ and $\l\in\R_{\geq0}$ we have that $\l s\in\s$. In the whole paper every cone will be polyhedral unless stated otherwise.

A cone is said to be \textbf{strongly convex} when it does not contain any non-trivial linear subspace. For a strongly convex polyhedral cone $\s$ a \textbf{vertex} of $\s$ is a one dimensional face of $\s$ or a vector generating such a one dimensional face. For a strongly convex cone $\s\subset \R^n$ we denote by $\P(\s)$ its image in $\P(\R^n)=\P^{n-1}(\R)$.

The \textbf{dimension} of a cone $\sigma$ is the dimension of the minimal linear subspace ${\mathcal L} (\sigma)$ containing $\sigma$ and is denoted by $\dim(\sigma)$.

The \textbf {dual} $\sigma^{\vee}$ of a cone $\sigma$ is the cone given by 
$$\sigma^{\vee} :=  \{ v \in \R^n \mid v \cdot u \geq 0, \forall u\in \sigma \}$$
where $u\cdot v$ stands for the dot product $(u_1,\ldots ,u_n)\cdot (v_1,\ldots ,v_n):= u_1v_1+\cdots +u_nv_n$.

\begin{lemma}\label{DualFuertConvDimMax}(see 1.2.(13) p. 14  \cite{Fulton})
	The dual of a polyhedral cone $\sigma$ has full dimension if and only if $\sigma$ is strongly convex. 
\end{lemma} 

The \textbf{relative interior} of a cone $\sigma$ is the interior of $\sigma$ as a subset of ${\mathcal L} (\sigma )$. That is, if $\s=\langle u^{(1)}, \ldots, u^{(k)} \rangle$ is a polyhedral cone:
$$\Int \langle u^{(1)},\ldots, u^{(s)}\rangle = \{ \lambda_1 u^{(1)} +\cdots + \lambda _s u^{(s)}, \, \lambda_i \in \R_{>0}\}.$$
A cone $\s$ is \textbf{open} if its interior, denoted by $\text{Int}(\s)$, is equal to $\s\backslash\{0\}$. A polyhedral cone different from $\{0\}$ is never open.

Let $S\subset\R^n$ be any subset. We will denote
\[
 S^\perp:= \{ v\in\R^n\mid u\cdot v=0, \forall u\in S\}.
\]
\begin{lemma}\label{lemma_int}
	Let $\sigma\subset\R^n$ be a polyhedral strongly convex cone. Given $\omega\in\R^n$,
	\[
		\omega\in \Int  (\sigma^\vee )\Leftrightarrow \sigma\subset {\langle\omega\rangle}^\vee\text{ and }\sigma\cap\omega^\perp=\{0\}.
	\]
\end{lemma}

\begin{proof}
Clearly if $\o\in  \Int  (\sigma^\vee )$ then $\o\in \s^\vee$ so $\s\subset \langle\o\rangle^\vee$.\\
Since $\s$ is strongly convex its dual cone $\s^\vee$ has full dimension. So the interior of $\s^\vee$ is its interior as a subset of $\R^n$.\\
 Then if $\o\cdot u=0$ for some $u\in\s$, for any $\e>0$ there exists $\o'\in\R^n$ such that $\|\o-\o'\|<\e$ and $\o'\cdot u<0$, hence $\o$ is not in $\Int  (\sigma^\vee )$.\\
 On the other hand if  $\o\cdot u>0$ for every $u\in \s$, then $\o\cdot u^{(i)}>0$ for every $i$  where $\{u^{(1)},\ldots,u^{(k)}\}$ is a set of generators of $\s$. Then for $\e>0$ small enough we have 
 $\o'\cdot u^{(i)}>0$ for every $i$ when $\o'\in\R^n$ satisfies $\|\o'-\o\|<\e$, hence the open ball $B(\o,\e)$ is in $\s^\vee$. This shows that $\o\in  \Int  (\sigma^\vee )$.
\end{proof}

\begin{lemma}\label{lemma_tech}
Let $\s$ be a full dimensional cone in $\R^n$ and $\g_1$, $\g_2\in \R^n$. Then
$$(\g_1+\s)\cap(\g_2+\s)\neq \emptyset.$$
\end{lemma}

\begin{proof}
Let $u^{(1)},\ldots, u^{(k)}\in\R^n$ be generators of $\s$. Since $\s$ is full dimensional the vector space spanned by the $u^{(i)}$ is $\R^n$. Thus there exist scalars $\l_i\in\R$ such that 
$$\g_1-\g_2=\sum_{i=1}^k\l_iu^{(i)}.$$
After a permutation of the $u^{(i)}$ we may assume that there exists an integer $l\leq k$ such that
$$\l_i\leq 0 \text{ for }i\leq l \text{ and } \l_i\geq 0 \text{ for } i>l.$$
Thus we have
$$\g_1+\sum_{i=1}^l(-\l_i)u^{(i)}=\g_2+\sum_{j=l+1}^k\l_ju^{(j)}\in (\g_1+\s)\cap(\g_2+\s).$$
\end{proof}

\begin{lemma}\label{SiEstaContenidoEnDosDesplazadosDeCono}
	Let  $\sigma_1$ and $\sigma_2$ be two cones and $\g_1$ and $\g_2$ be vectors of $\R^n$. Let us assume that $\s_1\cap\s_2$ is full dimensional. Then there exists a vector $\g\in\R^n$ such that

	$$\left( \g_1 +\sigma_1\right)\cap\left( \g_2 +\sigma_2\right)\subset \g + \sigma_1\cap\sigma_2.$$
	
\end{lemma}

\begin{proof}
By Lemma \ref{lemma_tech} there exists $\g\in(\g_1-\s_1\cap\s_2)\cap(\g_2-\s_1\cap\s_2)$. In particular we have that
$$\g_1,\, \g_2\in \g+\s_1\cap\s_2.$$
Thus 
$$\g_1+\s_1\subset \g+\s_1 \text{ and } \g_2+\s_2\subset \g+\s_2.$$
But 
$$(\g+\s_1)\cap(\g+\s_2)=\g+\s_1\cap\s_2.$$
This proves the lemma. 
\end{proof}

\section{Preorders}\label{preorders}
\begin{definition}(\cite{EI}; see also \cite{GT})
A preorder on $\R^n$ is a relation $\preceq$ satisfying the following conditions:
\begin{itemize}
\item[i)] For every $u$, $v\in\R^n$ we have $u\preceq v$ or $v\preceq u$.
\item[ii)] For every $u$, $v$, $w\in\R^n$ we have $u\preceq v, \ v\preceq w\Longrightarrow u\preceq w$.
\item[iii)] For every $u$, $v$, $w\in\R^n$, if $u\preceq v$ then $u+w\preceq v+w$.
\end{itemize}
\end{definition}

By ii) and iii) a preorder $\preceq$ is  {\bf compatible with the group structure}, i.e. $\a\preceq \b$ and $\g\preceq \delta$ implies $\a+\g\preceq \b+\delta$ for every $\a$, $\b$, $\g$ and $\delta\in\R^n$.

\begin{remark}
An order is a preorder if and only if it is a total order compatible with the group structure.
\end{remark}

Given a preorder $\preceq$ in $\R^n$ the set of non-negative elements will be denoted by $(\R^n)_{\succeq 0}$; that is,
\[
(\R^n)_{\succeq 0}:=\{ \alpha\in\R^n \mid 0\preceq \alpha\}.
\]
A set $S\subset\R^n$ is called {\bf $\preceq$-non-negative} when $S\subset (\R^n)_{\succeq 0}$. 

We will say that a preorder is {\bf positive} when the first orthant is non-negative for that preorder.

\begin{remark}\label{ElConoPosEsFuertConvexo}
	When a preorder $\preceq$ is a total order on $\R^n$, a $\preceq$-non-negative set does not contain any non trivial linear subspace. In particular a $\preceq$-non-negative cone is strongly convex. 
\end{remark}

\begin{lemma} \label{LaUniondeConosEstaEnUncono}
	Given a preorder $\preceq$ on $\R^n$, let $\sigma_1$ and $\sigma_2$ be $\preceq$-non-negative rational cones.  There exists a $\preceq$-non-negative rational cone $\sigma_3$ such that
  \[
  \sigma_1\cup\sigma_2\subset \sigma_3.
  \]
\end{lemma}
\begin{proof}
	Take $\sigma_3$ to be the cone generated by $\sigma_1\cup\sigma_2$. The elements of $\sigma_3$ are of the form $v_1+v_2$ with $v_i\in\sigma_i$. Since $v_i\succeq 0$ then $v_1+v_2\succeq 0$.
\end{proof}

\begin{lemma}\label{LaUniondeConosTrasladadosEstaEnUnconoTrasladado}
  Given a positive total order $\preceq$ on $\R^n$ compatible with the group structure, let $\sigma_1$ and $\sigma_2$ be $\preceq$-non-negative rational cones. For any two points $\gamma_1$ and $\gamma_2$ in $\R^n$ there exist $\gamma_3\in\R^n$ and a $\preceq$-non-negative rational cone $\sigma_3$ such that
  \[
  \left(\gamma_1 +\sigma_1\right)\cup\left(\gamma_2 +\sigma_2\right)\subset \gamma_3 +\sigma_3.
  \]
\end{lemma}

\begin{proof}
Let $\s_3$ denote a  non-negative rational cone containing $\s_1$, $\s_2$ and the first orthant (such a cone exists by Lemma \ref{LaUniondeConosEstaEnUncono}). Since $\sigma_3$ contains the first orthant it is full dimensional. By Lemma \ref{lemma_tech} we can pick an element $\g_3$ in $(\g_1-\s_3)\cap(\g_2-\s_3)$.\\
In particular  $\g_1-\g_3\in\s_3$. Since $\s_1\subset \s_3$ we have that 
$$\g_1+\s_1=\g_3+(\g_1-\g_3)+\s_1\subset \g_3+\s_3.$$
By symmetry we also have
$$\g_2+\s_2\subset \g_3+\s_3.$$
This proves the lemma.
\end{proof}


A vector $\omega\in\R^n$ induces a preorder in $\R^n$ denoted by $\leq_\o$ and defined as follows:
\[
\alpha\leq_\omega\beta \Longleftrightarrow \omega\cdot\alpha\leq\omega\cdot\beta
\]
where $\omega\cdot\alpha$ denotes the dot product.

An $s$-tuple $(u_1,\ldots ,u_s)\in \R^{ns}$ induces a preorder in $\R^n$ denoted by $\leq_{(u_1,\ldots ,u_s)}$ and defined as follows:
\begin{equation}\label{OrdenDadoPorVectores}
\alpha\leq_{(u_1,\ldots ,u_s)}\beta \Longleftrightarrow \texttt{p}_{u_1,\ldots u_s} (\alpha)\leq_{\text{lex}} \texttt{p}_{u_1,\ldots, u_s} (\beta )
\end{equation}

where $\texttt{p}_{u_1,\ldots , u_s} (u) := (u\cdot u_1,\ldots , u\cdot u_s)$  and $\leq_{\text{lex}}$ is the lexicographical order.

The following result is given in \cite[Theorem 2.5]{Robbiano:1986}:
\begin{theorem}\label{rob}
	Let $\preceq$ be a preorder  on $\Q^n$. Then there exist $u_1,\ldots , u_s$ vectors in $\R^n$, for some integer $1\leq s\leq n$, such that the map
	\[
		\emph{\texttt{p}}_{u_1,\ldots, u_s} :(\Q^n, \preceq)\rightarrow (\R^s, \leq_{\text{lex}})
	\]
	is an injective morphism of ordered groups.\\
	Moreover we may always assume that the $u_i$ are orthogonal and, when the preorder is a total order,  $s=n$.
\end{theorem}

In the light of Theorem \ref{rob}, when interested in restrictions to the rational numbers, we may consider only preorders of type (\ref{OrdenDadoPorVectores}). These preorders are called {\bf continuous preorders}. An order which is a continuous preorder is called a {\bf continuous order}.

The following lemma can be deduced from Theorem 3.4 given in \cite{Neu} but for the convenience of the reader we provide a direct proof of it:
\begin{lemma}\label{LosEnterosInterseccionUnConoEstaBienOrenado}
	Given a total order $\preceq$ in $\R^n$ compatible with the group structure, let $\sigma$ be a $\preceq$-non-negative rational cone. The set $\sigma\cap\Z^n$ is well ordered.
\end{lemma}
\begin{proof}
	
Let $\{v^{(1)},\ldots ,v^{(s)}\}\subset \Z^n$ be a system of generators of the semigroup $\sigma\cap\Z^n$ and consider the mapping 
\[
\begin{array}{cccc}
	\nu_{\preceq,\{v^{(1)},\ldots ,v^{(s)}\}}:
		& \K [y_1,\ldots ,y_s]
			&\longrightarrow
				& (\sigma\cap\Z^n, \preceq )\cup \{\infty\}\\
		& f(y_1, \ldots ,y_s)
			&\mapsto
				&\min_\preceq \Supp (f( x^{v^{(1)}}, \ldots ,x^{v^{(s)}}))
\end{array}
\]
where $x=(x_1,\ldots, x_n)$ denotes a vector of new indeterminates.
Since $\{v^{(1)},\ldots ,v^{(s)}\}$ generates $\sigma\cap\Z^n$, the map $\nu_{\preceq,\{v^{(1)},\ldots ,v^{(s)}\}}$ is surjective.

Consider the ring $R:=\frac{\K [y_1,\ldots ,y_s]}{I}$ where $I$ is the following ideal:
 $$I=\{ f\in \K [y_1,\ldots ,y_s]\mid f( x^{v^{(1)}}, \ldots ,x^{v^{(s)}})=0\}.$$  

Suppose that the set $(\sigma\cap\Z^n,\preceq)$ is not well ordered. Then there exists a sequence $(\gamma^{(i)})_{i\in\Z_{\geq 0}}\subset \sigma\cap\Z^n$ with $\gamma^{(i+1)}\preceq \gamma^{(i)}$ and $\gamma^{(i+1)}\neq \gamma^{(i)}$. Consider the ideals $J_i:= \{ f\in R\mid \nu_{\preceq,\{v^{(1)},\ldots ,v^{(s)}\}}f( x^{v^{(1)}}, \ldots ,x^{v^{(s)}})\succeq \gamma^{(i)}\}$. The chain $( J_i)_{i\in\Z_{\geq 0}}$ is an increasing sequence of ideals. Since any element of ${\nu_{\preceq,\{v^{(1)},\ldots ,v^{(s)}\}}}^{-1}(\gamma^{(i+1)})$ is not in $J_i$ we have that $J_i\neq J_{i+1}$ which contradicts the Noetherianity of $R$.
\end{proof}
A preorder $\preceq$ {\bf refines} $\preceq'$ when $\alpha\preceq\beta$ implies 
$\alpha\preceq'\beta$ for every $\a$, $\b\in\R^n$. For instance the preorder $\leq_{(u_1,\ldots ,u_s)}$ refines $\leq_{(u_1,\ldots ,u_k)}$ for every $s>k$ and every vectors $u_1,\ldots,u_s$.

\begin{lemma}\label{lemma_positive}
	Let be given $\omega\in\R^n\backslash\{0\}$ and a strongly convex cone $\sigma\subset\R^n$ with $\sigma\subset {\left<\omega\right>}^\vee$. There exists a continuous order $\preceq$ in $\R^n$ that refines $\leq_\omega$  such that $\sigma$ is a $\preceq$-non-negative set.
\end{lemma}

\begin{proof}
The proof is made by induction on $n$. For $n=1$, $\leq_\omega$ is a continuous  order in $\R$, $\s=\R_{\geq_\o 0}$ hence $\s$ is $\leq_\o$-non-negative.\\
Let us assume that the lemma is proven in dimension $n-1$ and let us consider $\o$ and $\s$ as in the statement of the lemma. After a linear change of coordinates we may assume that $\o=(0,\ldots,0,1)$. Then $\s'=\s\cap\langle \o\rangle^\perp$ is a strongly convex cone of $\langle \o\rangle^\perp\simeq\R^{n-1}$. \\
Let $\o'\in\langle \o\rangle^\perp$ be a nonzero vector such that $\s'\subset \langle\o'\rangle^\vee$. Such a vector $\o'$ exists since $\s'$ is strongly convex and it is included in a half-space. By the inductive hypothesis there exists a continuous  order $\preceq'$ in $\R^{n-1}\simeq\langle \o\rangle^\perp$ that refines the restriction of $\leq_{\o'}$ to $\langle\o\rangle^\perp$ and such that $\s'$ is $\preceq'$-non-negative.  Such an order $\preceq'$ is equal to $\leq_{(u_1,\ldots,u_s)}$ for some vectors $u_1,\ldots,u_s$ of $\langle \o\rangle^\perp$. 
Then $\leq_{(\o,u_1,\ldots,u_s)}$ is a continuous  order that refines $\leq_\o$ and $\s$ is $\leq_{(\o,u_1,\ldots,u_s)}$-non-negative.
\end{proof}


\begin{lemma}\label{lemma1}
	Let $u_1,u_2,\ldots, u_n$ be a basis of $\R^n$ and let $S$ be a subset of $\R^n$. Then
	\[
		S\subset \bigcap_{  \e_2,\ldots,\e_n\in\{-1,1\} } (\R^n)_{\geq_{(u_1,\e_2u_{2},\ldots ,\e_nu_n)} 0}
	\]
	holds if and only if $S\subset {\left< u_1\right>}^\vee$ and $S\cap {u_1}^\perp \subset\{0\}$.
\end{lemma} 

\begin{proof}
Let $s\in\R^n$ be written as
$$s=\l_1u_1+\cdots+\l_nu_n$$
where the $\l_i$ are real numbers.\\
If 
$$s\in \bigcap_{ \e_2,\ldots,\e_n\in\{-1,1\}} (\R^n)_{\geq_{(u_1,\e_2u_{2},\ldots ,\e_nu_{n})} 0}$$
then we have that $s\cdot u_1\geq 0$ so $s\in \langle u_1\rangle^{\vee}$. Moreover if $s\cdot u_1=0$ then 
$$(s\cdot \e_2u_2,\ldots, s\cdot \e_nu_n)\geq_{\text{lex}} 0 \ \ \ \ \forall (\e_2,\ldots,\e_n)\in\{-1,1\}^{n-1}.$$
Thus $s\cdot u_2\geq 0$ and $s\cdot (-u_2)\geq 0$, so $s\cdot u_2=0$. By induction we have $s\cdot u_k=0$ for every $k$, hence $s=0$ since $(u_1,\ldots,u_n)$ is a basis of $\R^n$. \\
On the other hand if 
$s\in S$, $S\subset {\left< u_1\right>}^\vee$ and $S\cap {u_1}^\perp \subset\{0\}$, then $s\cdot u_1\geq 0$. If $s\cdot u_1>0$ then $s\in (\R^n)_{\geq_{(u_1,\e_2u_{2},\ldots ,\e_nu_{n})} 0}$ for every $\e_i$. If $s\cdot u_1=0$ then $s=0$ by assumption, hence $s\in (\R^n)_{\geq_{(u_1,\e_2u_{2},\ldots ,\e_nu_{n})} 0}$ for every $\e_i$. This proves the equivalence.

\end{proof}

\begin{corollary}\label{PositParaTodoRef}
	Let $\omega\neq 0$ be a vector in $\R^n$ and let $\sigma\subset\R^n$ be a cone. Let $u_2,\ldots,u_n\in\R^n$ be such that $\o$, $u_2,\ldots, u_n$ form a basis of $\R^n$. Then the following properties are equivalent:
	
	\begin{enumerate}
	\item[i)] 
	$\displaystyle\sigma\subset \bigcap_{\e\in\{-1,1\}^{n-1} } (\R^n)_{\geq_{(\o,\e_2u_2,\ldots ,\e_nu_n)} 0}$

	\item[ii)] $\displaystyle\sigma\subset \bigcap_{\preceq\text{ continuous preorder  refining } <_\omega} (\R^n)_{\succeq 0}$
	
	\item[iii)] $\omega\in \Int  (\sigma^\vee )$.
	
	\end{enumerate}

\end{corollary}

\begin{proof}
Since $\geq_{(\o,\e_2u_2,\ldots ,\e_nu_n)}$ is a preorder that refines $\geq_\o$ we have that $ii) \Longrightarrow i)$.\\

Let us assume that 
$$\displaystyle \s\subset  \bigcap_{  \e_2,\ldots,\e_n\in\{-1,1\} } (\R^n)_{\geq_{(\o,\e_2u_{2},\ldots ,\e_nu_n)} 0}.$$
 Thus by Lemma \ref{lemma_int} and Lemma \ref{lemma1} we have that $\o\in \Int (\s^\vee)$. This shows that  $i) \Longrightarrow iii)$. \\
 
On the other hand if $\o\in \Int (\s^\vee)$ then for every $s\in\s\backslash\{0\}$ we have $\o\cdot s>0$ by Lemma \ref{lemma_int}, i.e. $0<_{\o} s$. Let $\preceq$ be a preorder refining $\leq_\o$. Then $s\preceq 0$ would imply that $s\leq_\o 0$ which is not possible. So necessarily we have that $s\succeq 0$. This shows $iii) \Longrightarrow  ii)$.
\end{proof}


\section{Algebraically  closed fields containing the field of formal power series}\label{alg_closed}

Let $\K$ be an algebraically closed characteristic zero field. A {\bf generalized Laurent series} in $n$ variables $\xi$ with rational exponents is a formal sum
\[
\xi =\sum_{\alpha\in\Q^n} \xi_\alpha x^\alpha
\]
whose coefficients $\xi_\a\in\K$. The {\bf support} of such a generalized Laurent series $\xi$ is the subset of $\R^n$ given by
\[
\Supp (\xi ):= \{ \alpha\in\Q^n\mid \xi_\alpha\neq 0\}.
\]

Given a total order $\preceq$ on $\Q^n$ which is compatible with the group structure the set of generalized Laurent series with $\preceq$-well-ordered support is an algebraically closed field (see for instance \cite{Ri2}). In this section we will describe a subfield of this field that is also algebraically closed.

A series $\xi$ is said to be a {\bf Laurent Puiseux} series when
\[
\Supp (\xi )\subset {\frac{1}{k}\Z}^n
\]
for some natural number $k$. When $k=1$ one simply says that the series is a {\bf Laurent series}.

Let $\sigma$ be a strongly convex rational cone. The set of Laurent series whose support is contained in $\sigma\cap\Z^n$ is a ring that will be denoted by $\K [[\sigma]]$. 

When $\sigma$ contains the first orthant, the ring $\K [[\sigma]]$ localized by the set of powers of $x_1\cdots x_n$ may be described in terms of support sets by
\[
	\K [[\sigma]]_{x_1\cdots x_n} =\left\{ \xi\mid \exists \gamma\in\Z^n  \text{  such that }		\Supp  (\xi )\subset (\gamma +\sigma )\cap\Z^n \right\}.
\]

Given a continuous positive  order $\preceq$ in $\R^n$, the union
\[
	\K [[\preceq]]:= \bigcup_{\sigma\ \preceq\text{-non-negative rational cone}}\K [[\sigma]]
\]
is a ring by Remark \ref{ElConoPosEsFuertConvexo} and Lemma \ref{LaUniondeConosEstaEnUncono}.

We can also describe  the localization $\K [[\preceq]]_{x_1\cdots x_n}$ in terms of support as:
$$
 	\K[[\preceq]]_{x_1\cdots x_n}=\left\{ \xi\mid \exists \gamma\in\Z^n , \sigma\subset (\R^n)_{\succeq 0}\text{ rational cone}, \  \Supp  (\xi )\subset (\gamma +\sigma)\cap\Z^n\right\}.
$$

\begin{definition}\label{field_fam}
Let $\K$ be a field and let $\Gamma$ be a totally ordered Abelian group. A collection of subsets $\mathcal{F}\subset\mathcal{P} (\Gamma )$ is a {\bf field family with respect to $\Gamma$} when the following properties hold:
\begin{enumerate}
	\item The set $\bigcup_{A\in\mathcal{F}} A$ generates $\Gamma$ as an Abelian group. \label{PropFieldFam1}
	\item The elements of $\mathcal{F}$ are well ordered. \label{PropFieldFam2}
	\item $A\in\mathcal{F},B\in\mathcal{F}$ implies $A\cup B\in\mathcal{F}$. \label{PropFieldFam3}
	\item $A\in\mathcal{F}, B\subset A$ implies $B\in\mathcal{F}$. \label{PropFieldFam4}
	\item $A\in\mathcal{F}$, $\gamma\in\Gamma$ implies $\gamma + A\in\mathcal{F}$. \label{PropFieldFam5}
	\item $A\in\mathcal{F}$, $A\subset \Gamma_{\geq 0}$ implies $\langle A\rangle \in\mathcal{F}$. \label{PropFieldFam6}
\end{enumerate}
\end{definition}

The concept of field family was introduced by F. J. Rayner in 1968 \cite{Rayner:1968}. This concept is used in \cite{Saavedra:2015} to extend McDonald's theorem \cite{McDonald:1995} to positive characteristic. The main use of  field families is the following 	
	 theorem:
	 
\begin{theorem}[Theorem 2 \cite{Rayner:1974}]\label{Rayner74}
	Let $\K$ be an algebraically closed field of characteristic zero, let $\Gamma$ be an ordered group and let $\Delta$ be the divisible envelope of $\Gamma$. Let $\mathcal{F}(\Delta )$ be any field family with respect to $\Delta$. The set of power series with coefficients in $\K$ whose support is an element of $\mathcal{F}(\Delta )$ is an algebraically closed field.
\end{theorem}


Given a continuous positive  order $\preceq$ in $\R^n$, consider the family $\mathcal{F}_\preceq \left(\Z^n\right)\subset\mathcal{P}\left(\Z^n\right)$ given by 
\[
\mathcal{F}_\preceq \left(\Z^n\right) := \left\{ A\subset \Z^n\mid  \exists \gamma\in\Z^n, \sigma\subset (\R^n)_{\succeq 0}\text{ rational cone}, \text{ with } A\subset \gamma +\sigma\right\}.
\]

With this notation we can write
\[
	\K [[\preceq]]_{x_1\cdots x_n}=\{\xi\mid \Supp  (\xi )\in \mathcal{F}_\preceq \left(\Z^n\right)\}.
\]
\begin{proposition}\label{EsFieldFamilyEnZ}
	Given a continuous positive order $\preceq$ in $\R^n$, the family $\mathcal{F}_\preceq\left(\Z^n\right)$ is a field family with respect to $\Z^n$.
\end{proposition}

\begin{proof} We have to check that the properties of Definition \ref{field_fam} are satisfied.\\
	Property \eqref{PropFieldFam1} follows from the fact that ${\Z_{\geq 0}}^n$ is an element of $\mathcal{F}_\preceq \left(\Z^n\right)$.
	Property \eqref{PropFieldFam2} has been shown in Lemma \ref{LosEnterosInterseccionUnConoEstaBienOrenado}.
	Property \eqref{PropFieldFam3} is direct consequence of Lemma \ref{LaUniondeConosTrasladadosEstaEnUnconoTrasladado}.
	Properties \eqref{PropFieldFam4} and \eqref{PropFieldFam5} follow from the  definition of $\mathcal{F}_\preceq \left(\Z^n\right)$. Property \eqref{PropFieldFam6} follows from the definition of a polyhedral cone and the fact that continuous orders respect the $\R$-vector space structure.
\end{proof}

Now consider the family $\mathcal{F}_\preceq \left(\Q^n\right)\subset\mathcal{P}\left(\Q^n\right)$ given by 
\[
\mathcal{F}_\preceq \left(\Q^n\right) := \left\{ A\subset \Q^n\mid  \exists k\in\Z_{> 0}  \text{ with } k A\in \mathcal{F}_\preceq \left(\Z^n\right) \right\}
\]
where $kA:= \{ k\alpha\mid \alpha\in A\}$.

\begin{proposition}\label{EsFieldFalimilyQ}
	Given a continuous positive  order $\preceq$ in $\R^n$, the family $\mathcal{F}_\preceq\left(\Q^n\right)$ is a field family with respect to $\Q^n$.
\end{proposition}
\begin{proof}
	Property \eqref{PropFieldFam1} of Definition \ref{field_fam} follows from the fact that, for all $k\in\Z_{> 0}$, ${{\frac{1}{k}\Z^n_{\geq 0}}}$ is an element of $\mathcal{F}_\preceq \left(\Q^n\right)$, and the set $\bigcup_{k\in\Z_{> 0}}{{\frac{1}{k}\Z^n_{\geq 0}}}$  generates $\Q^n$.\\
	The remaining properties follow directly from Proposition \ref{EsFieldFamilyEnZ}.
\end{proof}

Let $\mathscr{S}^\K_\preceq$  be the set of Laurent Puiseux series with coefficients in $\K$ whose support is an element of $\mathcal{F}_\preceq (\Q^n)$, i.e.

	\[
		\mathscr{S}^\K_\preceq =\left\{ \xi\mid \exists k\in \Z_{> 0}, \gamma\in\Z^n,  \sigma\subset (\R^n)_{\succeq 0} \text{ rational cone, } 		\Supp  (\xi )\subset (\gamma +\sigma )\cap \frac{1}{k}\Z^n \right\}.
	\]

Then we can state the main result of this part:

\begin{theorem}\label{TeoremaAlgCerrado}
 Let $\K$ be an algebraically closed field of characteristic zero. 	Given a continuous positive  order $\preceq$ in $\R^n$, the set $\mathscr{S}^\K_\preceq$ is an algebraically closed field.
\end{theorem}
\begin{proof}
	This is a direct consequence of Theorem \ref{Rayner74} and Proposition \ref{EsFieldFalimilyQ}.
\end{proof}

\begin{corollary}\label{SiEsDesoporteBienOrdenado}
	Let $P(T)\in\K [[x]][T]$ be a polynomial in $T$, let $\preceq$ be a continuous positive order in $\Q^n$ and let $\xi$ be a root of $P(T)$ in the field of Laurent Puiseux series with $\preceq$-well-ordered support. Then $\xi$ is an element of $\mathscr{S}^\K_\preceq$
\end{corollary}
\begin{proof}
 This is a direct consequence of Theorem \ref{TeoremaAlgCerrado} and the inclusion
 \begin{equation}\label{InclusionEnBienOrdenado}
 	\K [[x]]\subset \mathscr{S}^\K_\preceq\subset \mathcal{W}^\K(\Q^n,\preceq)
 \end{equation}
 
 where $\mathcal{W}^\K(\Q^n,\preceq)$ denotes field of Laurent Puiseux series with $\preceq$-well-ordered support.
\end{proof}


\begin{remark} Taking for $\preceq$ the order  $\leq_\omega$ with $\omega$ a vector with rationally independent coordinates, we recover the main theorems of \cite{McDonald:1995}, \cite{Gonzalez:2000} and \cite{ArocaIlardi:2009} as corollaries of Theorem \ref{TeoremaAlgCerrado}. If we take for $\preceq$ the order  $\leq_{(u_1,\ldots ,u_n)}$ where $u_1,\ldots ,u_n\in \Q^n$ are $\Q$-linearly independent, the main result in \cite{SotoVicente} is a particular case of Theorem \ref{TeoremaAlgCerrado}.
\end{remark}

\begin{remark}
Let $\preceq$ be a continuous positive  order in $\R^n$. Inclusion (\ref{InclusionEnBienOrdenado}) implies that the map
$$\nu_\preceq :\mathscr{S}^\K_\preceq \to \R^n\cup\{\infty\}$$
defined by $\nu_\preceq(\varphi)=\min \Supp\varphi$
 and $\nu_\preceq(0)=\infty$ is a valuation. 
\end{remark}

\begin{remark}\label{hens_valued_field}
Let $\preceq$ be a continuous positive  order in $\R^n$. By Theorem \ref{rob} there exist $s$ vectors in $\R^n$, with $s\leq n$, such that $\preceq = \leq _{(u_1,\ldots,u_s)}$. Then the map
$$\nu_{(u_1,\ldots,u_s)} :\mathscr{S}^\K_\preceq \to \R^s\cup\{\infty\}$$
defined by $\nu_{(u_1,\ldots,u_s)}(\varphi)=\min\{\texttt{p}_{u_1,\ldots,u_s}(\a)\mid a_\a\neq 0\}$ for $\xi=\sum_{\a\in\Q^n}\xi_\a x^\a\neq 0$ and $\nu_{(u_1,\ldots,u_s)}(0)=\infty$ is a valuation. \\
If $\preceq=\leq_{(v_1,\ldots,v_t)}$ for some vectors $v_1,\ldots,v_t$ then $\nu_{(u_1,\ldots,u_s)}$ and $\nu_{(v_1,\ldots,v_t)}$ are equivalent valuations. 

Since $\mathscr{S}_\preceq$ is an algebraically closed field then $(\mathscr{S}_\preceq,\nu_{(u_1,\ldots,u_s)})$ is a Henselian valued field and its value group is an ordered subgroup of $(\R^s,\geq_{\text{lex}})$.
\end{remark}


\section{The maximal dual cone}\label{dual_cone}

In this part $\sigma$ will denote a strongly convex cone containing the first orthant,   $\xi\in \K [[\sigma]]_{x_1\cdots x_n}$ will be algebraic over $\K[[x]]$ where $\K$ is a characteristic zero field. We denote by $P\in \K[[x]][T]$  the minimal polynomial of $\xi$ and, for any continuous positive order $\preceq$, let $\xi_1^\preceq,\ldots, \xi_d^\preceq$ denote the roots of $P(T)$ in $\mathscr{S}^\K_\preceq$.  We set
	\[
		\tau_0 (\xi ):= \left\{\omega\in {\R_{> 0}}^n\mid \text{ for all }\preceq\text{ that refines }\leq_\omega ,\ \exists i \text{ such that } \xi = \xi_i^\preceq\right\}
	\]
		\[
		\tau_1 (\xi ):= \left\{\omega\in {\R_{> 0}}^n\mid\xi \neq\xi_i^\preceq, \text{ for all }\preceq\text{ that refines }\leq_\omega, \ \forall i= 1,\ldots ,d \right\}
	\]

The aim of this part is to prove the second main result of this work. This one states that $\t_0(\xi)$ is the "maximal dual cone" of $\Supp(\xi)$ in the following sense (see also Lemma \ref{IntDeDualEnTau0}): when $\xi\notin\K[[x]]_{x_1\cdots x_n}$, for every $\o\in {\R_{> 0}}^n$ belonging to the boundary of $\t_0(\xi)$ there exists $k\in\R$ such that all but a finite number of elements of $\Supp(\xi)$ will be in the set $\{\a\in\Z^n\mid \a\cdot\o\geq k\}$ and the set $\{\a\in\Z^n\mid \a\cdot\o=k\}$ contains an infinite number of elements of $\Supp(\xi)$. The strategy of the proof is based on the fact that $\t_0(\xi)$ and $\t_1(\xi)$ are disjoint open subsets of ${\R_{>0}}^n$ and on the characterizations of $\t_0(\xi)$ and $\t_1(\xi)$ given in  Lemmas \ref{characterization_tau_0} and \ref{characterization_tau_1}.\\

\begin{lemma}\label{IntDeDualEnTau0}
	Let $\sigma$ be a strongly convex cone containing the first orthant and let  $\xi\in \K [[\sigma]]_{x_1\cdots x_n}$ be algebraic over $\K[[x]]_{x_1\cdots x_n}$. Then we have that
	\[
		\int  (\sigma^\vee)\subset \tau_0(\xi ).
	\]

\end{lemma}

\begin{proof}
By Lemma \ref{DualFuertConvDimMax}  we have that $\int ({\sigma}^\vee )=\Int({\sigma}^\vee )$.
	Given $\omega\in \Int  (\sigma^\vee)$, by Corollary \ref{PositParaTodoRef} we have that $\sigma\subset \bigcap_{\preceq\text{ refines } \leq_\omega} {(\R^n)}_{\succeq 0}$. Hence $\xi\in\mathscr{S}_\preceq$ is a root of $P$ for any order $\preceq$ that refines $\leq_\omega$.
\end{proof}

\begin{lemma}\label{SiEstaEntodoslosanillosquerefinan}
	Let $\xi$ be a Laurent series and let $\omega$ be a non-zero vector in $\R^n$. Suppose that $\xi\in \mathscr{S}_\preceq$ for any continuous positive order $\preceq$ refining $\leq_\omega$. Then there exists a cone $\sigma_0$ and $\gamma_0\in\Z^n$ such that $\omega\in\Int ({\sigma_0}^\vee)$ and $\Supp (\xi )\subset \gamma_0 +\sigma_0$.
\end{lemma}
\begin{proof}
	 Let $u_2,\ldots, u_n\in\R^n$ be a basis of $\omega^\perp$. For any $\e=(\e_2,\ldots,\e_n)\in\{-1,1\}^{n-1}$  the series $\xi$ is an element of $\mathscr{S}_{\leq_{\left(\omega,\e_2u_{2},\ldots ,\e_nu_n\right)}}$. That is there exist $\gamma_\e\in\Z^n$ and a $\leq_{\left(\omega,\e_2u_{2},\ldots ,\e_nu_n\right)}$-non-negative cone $\sigma_\e$ with
	\[
		\Supp (\xi)\subset \gamma_\e +\sigma_\e. 
	\]
	Let $\s'$ be  the cone generated by the following $2(n-1)$ vectors:
	$$\o+u_i,\ \o-u_i\text{ for } i=2,\ldots,n.$$
	This cone is full dimensional since the vectors $\o$, $u_2,\ldots, u_n$ form a basis of $\R^n$. Moreover, for every $i$, we have 
	$$\texttt{p}_{\o,u_2,\ldots,u_n}(\o\pm u_i)=(\o\cdot\o, \pm u_2\cdot u_i,\ldots, \pm u_n\cdot u_i)>_{\text{lex}} \texttt{p}_{\o,u_2,\ldots,u_n}(0)=(0,\ldots,0)$$
	since $u_i$ is orthogonal to $\o$ for all $i$ and $\o\cdot\o>0$, thus $\s'$ is $\leq_{\left(\omega,\e_2u_{2},\ldots ,\e_nu_n\right)}$-non-negative. By replacing $\s_\e$ by $\s_\e+\s'$ we may assume that $\s_\e$ contains the cone $\s'$.
	
	Set $\sigma_0:=\bigcap_{\e\in\{-1,1\}^{n-1}}\sigma_\e$. Since the intersection of the $\s_\e$ contains $\s'$ which is full dimensional, by Lemma \ref{SiEstaContenidoEnDosDesplazadosDeCono} 
	 there exists $\gamma_0\in\Z^n$ such that 
	 \[
		\Supp (\xi)\subset \gamma_0 +\sigma_0.
	\]
	
\end{proof}


\begin{lemma}\label{Tau0Open}
		Let $\sigma$ be a strongly convex cone containing the first orthant and let  $\xi\in \K [[\sigma]]_{x_1\cdots x_n}$ be algebraic over $\K[[x]]_{x_1\cdots x_n}$. Then the set $\tau_0 (\xi)$ is open.
\end{lemma}

\begin{proof}
	If $\omega\in \tau_0 (\xi )$ then $\xi\in\mathscr{S}_\preceq$ for all continuous positive order $\preceq$ that refines $\leq_\omega$. By Lemma \ref{SiEstaEntodoslosanillosquerefinan} there exists 
	 a cone $\sigma_0$ and $\gamma_0\in\Z^n$ such that $\omega\in\Int ({\sigma_0}^\vee)$ and
	 \[
		\Supp (\xi)\subset \gamma_0 +\sigma_0.
	\]

	Since $\s_0$ is  a $\leq_{\left(\omega,\e_2u_{2},\ldots ,\e_nu_n\right)}$-non-negative cone for every $\e\in\{-1,1\}^{n-1}$, by Corollary  \ref{PositParaTodoRef}, $\o\in\int(\s_0^\vee)$. 		
	But $\int(\s_0^\vee)\subset\t_0(\xi)$ by Lemma \ref{IntDeDualEnTau0} hence $\t_0(\xi)$ is open.
\end{proof}


 \begin{lemma}\label{compacity_lemma}
Let $0\leq l\leq n$ be an integer and  $u_1,\ldots,u_l$ be nonzero orthogonal vectors of $\R^n$ such that $\leq_{(u_1,\ldots,u_l)}$ is a continuous positive preorder. Then there exists a finite set $\T$ of strongly convex rational cones of $\R^n$ such that for any positive continuous order $\preceq$ refining $\leq _{(u_1,\ldots,u_l)}$ there exists a $\preceq$-non-negative cone $\t\in \T$ such that  the roots of $P(T)$ in $\mathscr{S}^\K_{\preceq}$ are in $\K[[\t]]_{x_1\cdots x_n}$ (by convention when $l=0$ any continuous positive order is refining $\leq _{\emptyset}$).\\
 Moreover if $\T$ is minimal with the previous property, then for any $u_l'$ close enough to $u_l$ and such that $\leq_{(u_1,\ldots,u_{l-1},u'_l)}$ is positive, the set $\T$ satisfies the same property for the sequence $(u_1,\ldots,u_{l-1},u'_l)$.
 \end{lemma}
 
 \begin{proof}
 The proof is made by a decreasing induction on $l$. For $l=n$ the lemma is a direct consequence of Theorem \ref{TeoremaAlgCerrado} since $\leq_{(u_1,\ldots,u_n)}$ is the only continuous order refining itself.\\
 Let $l<n$ and let us assume that the theorem is proven for the integer $l+1$. Let $u_1,\ldots,u_l$ be orthogonal vectors  of $\R^n$  such that $\leq_{(u_1,\ldots,u_l)}$ is positive and let $v\in \langle u_1,\ldots,u_l\rangle^\perp$ such that $\leq_{(u_1,\ldots,u_l,v)}$ is positive. By the inductive assumption there exists a finite set of strongly convex rational cones $\T_v$ such that for every positive continuous order $\preceq$ refining $\leq_{(u_1,\ldots,u_l,v)}$ there exists a cone $\t\in \T_v$ such that $\t$ is $\preceq$-non-negative and the roots of $P(T)$ in $\mathscr{S}^\K_{\preceq}$ are in $\K[[\t]]_{x_1\cdots x_n}$.\\
 Let $\preceq$ be a positive continuous order refining $\leq_{(u_1,\ldots,u_l,v)}$ and let $\t\in \T_v$ such that $\t$ is $\preceq$-non-negative and the roots of $P(T)$ in $\mathscr{S}^\K_{\preceq}$ are in $\K[[\t]]_{x_1\cdots x_n}$. By  Theorem \ref{rob} there exist orthogonal vectors $w_1,\ldots,w_{n-l-1}\in\langle u_1,\ldots,u_l,v\rangle^\perp$ such that
 $$\preceq\,=\, \leq_{(u_1,\ldots,u_l,v,w_1,\ldots,w_{n-l-1})}.$$
 Since this order does not change if we replace $w_1$ by $\mu w_1$ where $\mu$ is a positive real number, we may assume that $\Vert w_1\Vert=1$.

 Set $\t':=\t\cap \langle u_1,\ldots,u_{l}\rangle^\perp$ and let $t_1,\ldots,t_s$ be vectors generating  the semigroup $\t'$ (we assume that none of these vectors are zero). We consider two cases:
 
 \begin{itemize}[leftmargin=*]
\item[1)]
  If $t_i\cdot v>0$ for every $i$ then $t_i\cdot v'>0$ for every $v'\in\R^n$ with $\Vert v-v'\Vert\ll 1$. Thus  $\t$ is $\preceq'$-non-negative for every positive continuous order $\preceq'$ that refines $\leq_{(u_1,\ldots,u_l,v')}$ for every $v'\in\R^n$ with $\Vert v-v'\Vert\ll 1$.
  \item[2)] 
   Let us assume that $t_i\cdot v=0$ for some index $i$. In this case for every index $i$ such that $t_i\cdot v=0$  we denote by $r(i)\leq n-l-1$ the smallest integer satisfying
   $$t_i\cdot w_{r(i)}>0.$$
   Such an integer $r(i)$ exists
since $\t$ is $\preceq$-non-negative, and $\preceq$ is a total order.\\
Firstly we claim there exists $\l_0>0$ such that $\t$ is $\preceq'$-non-negative for every positive continuous order $\preceq'$ refining 
$$\leq_{(u_1,\ldots,u_l,v+\l w_1,w_2,\ldots, w_{n-l-1})}$$
for any positive real number $\l<\l_0$.\\
We first prove that $\leq_{(u_1,\ldots,u_l,v+\l w_1,w_2,\ldots, w_{n-l-1})}$ is a continuous positive preorder for $\l$ small enough: we only have to prove that $e^{(j)}\geq_{(u_1,\ldots,u_l,v+\l w_1,w_2,\ldots, w_{n-l-1})} 0$ for every $j$ if $\l$ is small enough where the $e^{(j)}$ are the vectors of the canonical basis of $\R^n$. The only problem may occur when $e^{(j)}\cdot u_i=0$ for every $i$, $e^{(j)}\cdot v>0$ and $e^{(j)}\cdot w_1<0$. But in this case we have $e^{(j)}\cdot(v+\l w_1)>0$ as soon as $\l< e^{(j)}\cdot v$ since $\Vert w_1\Vert=\Vert e^{(j)}\Vert=1$. Thus for every positive number $\l<\l_1:=\min_j\{e^{(j)}\cdot v\}$, where $j$ runs over the indices such that $e^{(j)}\cdot v>0$, the continuous preorder $\leq_{(u_1,\ldots,u_l,v+\l w_1,w_2,\ldots, w_{n-l-1})}$ is positive.\\
Then let $\l$ be a positive real number satisfying
$$0<\l< \l_{w_1,\t}:=\min_i \left\{ \frac{t_i\cdot v}{| t_i\cdot w_1|}\right\}\in\R_{>0}\cup\{+\infty\}$$
where the minimum is taken over all the indices $i$ such that $t_i\cdot v>0$. This is an abuse of notation since $\l_{w_1,\t}$ may depend on the generators $t_i$ of the semigroup $\t'$. \\
Then if $t_i\cdot v>0$ for some integer $i$ we have that $t_i\cdot(v+\l w_1)>0$  by definition of $\l_{w_1,\t}$. 
If $t_i\cdot v=0$ for some integer $i$ then we have 
$$t_i\cdot (v+\l w_1)=t_i\cdot \l w_1=\cdots=t_i\cdot w_{r(i)-1}=0 \text{ and } t_i\cdot w_{r(i)}>0.$$
This shows that the cone $\t$ is $\preceq'$-positive for every positive continuous order $\preceq'$ refining 
$$\leq_{(u_1,\ldots,u_l,v+\l w_1,w_2,\ldots, w_{n-l-1})}$$ 
as soon as $\l<\l_0:=\min\{\l_1,\l_{w_1,\t}\}$ and the claim is proven.\\
Now let us set 
$$\l_2=\inf_{w_1,\tilde\t} \l_{w_1,\tilde\t}$$
where $w_1$ runs over all unit vectors of $\langle u_1,\ldots,u_l,v\rangle^\perp$ and all $\tilde\t\in \T_v$. As mentioned before $\l_{w_1,\tilde\t}$ depends on the choice of generators $\tilde t_i$ of the semigroup $\tilde\t\cap \langle u_1,\ldots,u_{l}\rangle^\perp$, so we assume that for each cone $\tilde\t\in\T_v$ we have chosen a set of generators $\tilde t_i$ of the semigroup $\tilde\t\cap \langle u_1,\ldots,u_{l}\rangle^\perp$ that allows us to define $\l_{w_1,\tilde\t}$. Then $\l_2>0$ since for every unit vector $w_1$ we have that
$$ \frac{\tilde t_i\cdot v}{| \tilde t_i\cdot w_1|}\geq \frac{\tilde t_i\cdot v}{\Vert\tilde t_i\Vert},$$
hence $\displaystyle \l_2>\min_{\tilde\t,i}\left\{\frac{\tilde t_i\cdot v}{\Vert \tilde t_i\Vert}\right\}>0$ where $\tilde\t$ runs over $\T_v$  and $i$ runs over the indices such that $\tilde t_i\cdot v>0$ where $\tilde t_1,\ldots,\tilde t_s$  are nonzero generators of $\tilde \t\cap \langle u_1,\ldots,u_{l}\rangle^\perp$.\\
This shows that for every $w'_1\in \langle u_1,\ldots,u_l,v\rangle^\perp$ small enough (i.e. with $\Vert w'_1\Vert<\l_3:=\min\{\l_1,\l_2\}$ - here $w'_1=\l w_1$ with $\l<\l_3$) and every positive continuous order $\preceq$ refining $\leq_{(u_1,\ldots,u_l,v+w'_1)}$ there is a cone $\tilde\t\in T_v$ such that $\tilde\t$ is $\preceq$-non-negative. 
\end{itemize}

These two cases show that for any vector $v'\in\langle u_1,\ldots,u_l\rangle^{\perp}$, with $\Vert v-v'\Vert$ small enough and such that $\leq_{(u_1,\ldots,u_l,v')}$ is positive, we can choose $\T_{v'}=\T_v$. \\
The existence of continuous orders  refining $\leq_{(u_1,\ldots,u_l)}$ which are not positive is equivalent to say  that the set $J$ of indices $j$ such that 
$$e^{(j)}\cdot u_1=\cdots=e^{(j)}\cdot u_l=0$$
 is not empty. So for a vector $v\in\langle u_1,\ldots,u_l\rangle^\perp$ the order $\leq_{(u_1,\ldots,u_l,v)}$ is positive if and only if $e^{(j)}\cdot v\geq 0$ for all $j\in J$, i.e. if $v\in \langle e^{(j)}, j\in J\rangle^\vee$. But $\P(\langle e^{(j)},j\in J\rangle^\vee)$ is compact.  Thus we can find a finite set $\mathcal V$ of vectors of $\langle u_1,\ldots,u_l\rangle^{\perp}$ such that for every $v'\in \langle u_1,\ldots,u_l\rangle^{\perp}$, such that $\leq_{(u_1,\ldots,u_l,v')}$ is positive,  there exists a $v\in \mathcal V$ such that $\T_{v'}=\T_v$. \\
Then the set $\T:=\bigcup_{v\in\mathcal V}\T_v$ suits for the sequence $(u_1,\ldots,u_l)$. This proves the lemma by induction.

  \end{proof}
  
  
  \begin{corollary}\label{cor1}
  
		Let $\sigma$ be a strongly convex cone containing the first orthant and let  $\xi\in \K [[\sigma]]_{x_1\cdots x_n}$ be algebraic over $\K[[x]]$. Then the set $\tau_1 (\xi )$ is open.
  \end{corollary}
  
  \begin{proof}
  Let $\o\in \t_1(\xi)$. Then $\xi\neq \xi_i^\preceq$ for every continuous positive order $\preceq$ refining $\leq_\o$ and for every $i$. Let $\T$ be a set of strongly convex cones for the sequence $\o$ satisfying Lemma \ref{compacity_lemma}  (here we apply this lemma with $l=1$). Then for every $\o'$ in a small neighborhood of $\o$ the set $\T$ is also a set of strongly convex cones satisfying the previous lemma for the sequence $\o'$. This implies that for every continuous positive order $\preceq'$ refining $ \leq_{\o'}$, there exists a continuous order $\preceq$ refining $\leq_\o$ such that $\xi_i^{\preceq'}=\xi_i^{\preceq}$ for every $i$. In particular for every continuous positive order $\preceq'$ refining $\leq_{\o'}$ and for every $i$ we have $\xi\neq\xi_i^{\preceq'}$. Thus $\o'\in\t_1(\xi)$ and $\t_1(\xi)$ is open.
  \end{proof}
  
  \begin{lemma}\label{initial_rational}
Let $\xi\in\K[[\s]]_{x_1\cdots x_n}$  (not necessarily algebraic over $\K[[x]]$) where $\s$ is a strongly convex rational cone and $\o\in\s^\vee$. Then the set
$$\{\a\cdot\o\mid \a\in\Supp(\xi)\}$$
has a minimum.
\end{lemma}

\begin{proof}
By Lemma \ref{lemma_positive}, there is a continuous positive order $\preceq$ refining $\leq_\o$ such that $\s$ is $\preceq$-non-negative. Thus, by Lemma \ref{LosEnterosInterseccionUnConoEstaBienOrenado}, $\s\cap\Z^n$ is well ordered for $\preceq$. Since $\Supp(\xi)\subset\g+\s$ it has $\preceq$-minimum. Let $\beta$ be the $\preceq$-minimum of $\Supp(\xi)$. Then $\omega\cdot\beta$ is the minimum of the set $\{\a\cdot\o\mid \a\in\Supp(\xi)\}$.
\end{proof}

\begin{definition}
Let $\xi\in\K[[\s]]_{x_1\cdots x_n}$ (not necessarily algebraic over $\K[[x]]$) where $\s$ is a strongly convex rational cone and $\o\in\s^\vee$. We write $\xi=\sum_\a\xi_\a x^\a$. The $\o$-order of $\xi$ is defined as
$$\nu_\o(\xi):=\min_{\a\in\Supp(\xi)}\{\a\cdot\o\}$$
and the $\o$-initial part of $\xi$ is defined by
$$\In_\o(\xi):=\sum_{\a\mid \a\cdot\o=\nu_\o(\xi)}\xi_\a x^\a.$$
\end{definition}


\begin{lemma}\label{characterization_tau_0}
Let $\sigma$ be a strongly convex cone containing the first orthant and let  $\xi\in \K [[\sigma]]_{x_1\cdots x_n}$ be algebraic over $\K[[x]]$. Then we have that
	$$\tau_0(\xi )= \left\{ \omega\in {\R_{> 0}}^n\mid \# \left(\Supp (\xi )\cap \left\{  u\in\R^n\mid u\cdot\omega\leq k\right\}\right)<\infty, \forall k\in\R\right\}.$$
\end{lemma}

\begin{proof}
Let $\o\in\t_0(\xi)$. This means that $\xi\in\mathscr{S}^\K_\preceq$ for all continuous positive order $\preceq$ that refines $\leq_\o$. By Lemma \ref{SiEstaEntodoslosanillosquerefinan}, there exists a cone $\sigma_0$ and $\gamma_0\in\Z^n$ such that $\omega\in\Int ({\sigma_0}^\vee)$ and $\Supp (\xi )\subset \gamma_0 +\gamma_0$.

 In particular 
we have that
$$\Supp(\xi)\cap\left\{  u\in\R^n\mid u\cdot\omega\leq k\right\}\subset \g+\{u\in\s \mid u\cdot\o\leq k-\g\cdot\o\}\cap\Z^n.$$
But such a set $\{u\in\s \mid u\cdot\o \leq l\}\cap\Z^n$ is a finite set for any $l$ since  $\s\subset \langle\o\rangle^\vee$ and $\s\cap\o^\perp=\{0\}$. This proves that 
$$\tau_0(\xi )\subset \left\{ \omega\in {\R_{> 0}}^n\mid \# \left(\Supp (\xi )\cap \left\{  u\in\R^n\mid u\cdot\omega\leq k\right\}\right)<\infty, \forall k\in\R\right\}.$$\\

On the other hand, let $\o\in{\R_{>0}}^n$ be such that 
\begin{equation}\label{ecuacionfinitoentriangulo}
	\#\left(\Supp (\xi )\cap \left\{  u\in\R^n\mid u\cdot\omega\leq k\right\}\right)<\infty, \ \forall k\in\R
\end{equation}
and let us consider an order $\preceq$ that refines $\leq_\o$.\\
By (\ref{ecuacionfinitoentriangulo}) we have that $\Supp (\xi)$ is $\preceq$-well-ordered, then, by Corollary \ref{SiEsDesoporteBienOrdenado}, $\xi$ is an element of $\mathscr{S}^\K_\preceq$. This shows that $\o\in\t_0(\xi)$.\\
\end{proof}


\begin{corollary}\label{t_0}
The set $\tau_0(\xi)$ is a (non polyhedral) full dimensional convex cone.
\end{corollary}

\begin{proof}
If $\o\in\t_0(\xi)$ then clearly $\l\o\in\t_0(\xi)$ for every $\l>0$ so $\t_0(\xi)$ is a cone. Moreover if $\o$, $\o'\in\t_0(\xi)$ we have that
$$\{u\in\R^n\mid u\cdot(\o+\o') \leq k+l\}\subset \{u\in\R^n\mid u\cdot\o \leq k\}\cup \{u\in\R^n\mid u\cdot\o' \leq l\} $$
hence $\o+\o'\in\t_0(\xi)$ by Lemma \ref{characterization_tau_0}. Hence $\t_0(\xi)$ is a convex cone.\\
Since $\s$ is strongly convex we have that $\s^\vee$ is full dimensional by Lemma \ref{DualFuertConvDimMax}. Hence by Lemma \ref{IntDeDualEnTau0}
  $\t_0(\xi)$ is full dimensional.

\end{proof}


 \begin{corollary}\label{proper_tau_0}
  We have that $\t_0(\xi)={\R_{>0}}^n$ if and only if $\Supp(\xi)\subset \g+{\R_{\geq 0}}^n$ for some vector $\g\in\R^n$.
  \end{corollary}
  
  \begin{proof}
 We remark that if $\Supp(\xi)\subset \g+ {\R_{>0}}^n$ then any $\o\in {\R_{>0}}^n$ will satisfy
 $$ \# \left(\Supp (\xi )\cap \left\{  u\in\R^n\mid u\cdot\omega\leq k\right\}\right)<\infty\ \ \ \  \forall k\in\R$$
 so ${\R_{>0}}^n=\t_0(\xi)$ by Lemma \ref{characterization_tau_0}.\\
  On the other hand let us assume that $\t_0(\xi)={\R_{>0}}^n$. Let $e^{(j)}$ be a vector of the canonical basis of $\R^n$. By Lemma \ref{compacity_lemma} there exists a finite set $\T_j$ of strongly convex cones of $\R^n$ such that for any  continuous positive order $\preceq$  refining $\leq_{e^{(j)}}$ there exists a $\preceq$-non-negative cone $\t\in \T_j$ such that the roots of $P$ in $\Sc$ are in $\K[[\t]]_{x_1\cdots x_n}$. Let $\o\in {\R_{>0}}^n$ be a vector such that $\Vert \o-e^{(j)}\Vert $ is small enough in order to ensure that $\T_j$ is also a set satisfying Lemma \ref{compacity_lemma} for $\o$ (i.e we apply it with $l=1$ and $u_1=\o$). Since $\o\in\t_0(\xi)$ there exists a cone $\t_j\in \T_j$ such that $\Supp(\xi)\subset \g_j+\t_j$ for some vector $\g_j$. This being true for $j=1,\ldots, n$ we have 
  $$\Supp(\xi)\subset \bigcap_{j=1}^n(\g_j+\t_j)\subset \g+\bigcap_{j=1}^n\t_j$$
  for some $\g\in\R^n$ by Lemma \ref{SiEstaContenidoEnDosDesplazadosDeCono}.  But $\bigcap_{j=1}^n\t_j$ is a $\leq_{e^{(j)}}$-non-negative cone for every $j=1,\ldots,n$, thus  $\bigcap_{j=1}^n\t_j\subset {\R_{\geq 0}}^n$. Hence $\Supp(\xi)\subset \g+{\R_{\geq 0}}^n$.

  \end{proof}


\begin{lemma}\label{characterization_tau_1}
Let $\sigma$ be a strongly convex cone containing the first orthant and let  $\xi\in \K [[\sigma]]_{x_1\cdots x_n}$ be algebraic over $\K[[x]]$. Then
	$$\tau_1(\xi )= \left\{ \omega\in {\R_{> 0}}^n\mid \# \left(\Supp (\xi )\cap \left\{  u\in\R^n\mid u\cdot\omega\leq k\right\}\right)=\infty, \forall k\in\R\right\}.$$
\end{lemma}

\begin{proof}
Let $\o_0\notin \t_1(\xi)$, i.e. there exist an integer $i$ and a continuous positive order $\preceq$ that refines $\leq_{\o_0}$ such that $\xi=\xi_i^{\preceq}$. Thus there exist $\g\in\Z^n$ and a $\preceq$-positive cone $\s'$ such that 
$\Supp(\xi)\subset \g+\s'$. Set $k_0:=\g\cdot\o_0$.  Then for every $k<k_0$ we have that
$$\Supp(\xi)\cap\{u\in\R^n\mid u\cdot\o_0\leq k\}=\emptyset,$$
hence 
$$\o_0\notin \left\{ \omega\in {\R_{> 0}}^n\mid \# \left(\Supp (\xi )\cap \left\{  u\in\R^n\mid u\cdot\omega\leq k\right\}\right)=\infty, \forall k\in\R\right\}.$$ This proves that
$$\tau_1(\xi )\supseteq\left\{ \omega\in {\R_{>0}}^n\mid \# \left(\Supp (\xi )\cap \left\{  u\in\R^n\mid u\cdot\omega\leq k\right\}\right)=\infty, \forall k\in\R\right\}.$$\\
Now let $\o\in{\R_{>0}}^n$ be  such that for some $k\in\R$
$$\# \left(\Supp(\xi)\cap\left\{  u\in\R^n\mid u\cdot\omega\leq k\right\}\right)<\infty.$$
Then for some real number $k_0$ we have that
$$\Supp(\xi)\cap\left\{  u\in\R^n\mid u\cdot\omega\leq k_0\right\}=\emptyset.$$
Let $\g_0\in\R^n$ be such that $\g_0\cdot\o\leq k_0$. Thus $\Supp(\xi)\subset \g_0+\langle\o\rangle^\vee$.\\
By assumption $\Supp(\xi)\subset \g+\s$ for some vector $\g$ and the strongly convex cone $\s$ given in the assumptions.
By Lemma \ref{SiEstaContenidoEnDosDesplazadosDeCono} there is a vector $\g'$ such that
$$(\g+\s)\cap(\g_0+\langle \o\rangle^\vee)\subset \g'+\s\cap\langle\o\rangle^\vee.$$
Thus we may assume that $\s\subset \langle \o\rangle^\vee$ and by Lemma \ref{lemma_positive} there exists a continuous positive order $\preceq$ in $\R^n$ that refines $\leq_\o$ such that $\s$ is a $\preceq$-non-negative cone. This proves that $\xi\in \Sc$ hence $\xi=\xi_i^\preceq$ for some $i$ and $\o\notin\t_1(\xi)$. Hence the reverse inclusion is proven.
\end{proof}

\begin{lemma}\label{lemma_ref}
Let $\xi$ be a Laurent series algebraic over $\K(\!(x)\!)$ and take $\o\in{\R_{>0}}^n$. If 
$$\forall \l<0\ \ \ \ \#\left(\Supp(\xi)\cap\{x\in\R^n \mid  \o\cdot x<\l\}\right)<\infty$$
then
$$\#\left(\Supp(\xi)\cap\{x\in\R^n \mid  \o\cdot x<0\}\right)<\infty.$$
\end{lemma}

\begin{proof}
Let $\mathcal L$ be the linear subspace of $H_\o:=\{x\in\R^n\ \mid\ \o\cdot x=0\}$ generated by the vectors with rational coordinates, and let us set $\mathcal L^\Q=\mathcal L\cap \Q^n$ . We have
\begin{equation}\label{eeqq}\mathcal L^\perp\cap H_\o\cap\Q^n=\{0\}.\end{equation}
Let $\pi$ denote the orthogonal projection onto $\mathcal L^\perp$ and whose kernel is $\mathcal L$. Since $\mathcal L\oplus \mathcal L^\perp=\R^n$, for every $x\in\R^n$ we can write in a unique way
$$x=y+z \text{ with } y\in \mathcal L \text{ and } z\in\mathcal L^\perp.$$
Then $\pi(x):=z$.\\
Since $\mathcal L$ is generated by vectors with rational coordinates, $\mathcal L$ is generated by $\mathcal L^\Q$ and we have that $\mathcal L^\perp$ is generated by the orthogonal of $\mathcal L^\Q$ in $\Q^n$. Thus for every vector $u$ of $\Q^n$ we have that $\pi(u)\in \Q^n$. Thus, if we denote by $e_1$, \ldots, $e_n$ the vectors of the canonical basis of $\Z^n$, $\pi(e_i)\in\Q^n$ for every $i$. Let $k\in\Z_{> 0}$ be an integer such that $k\pi(e_i)\in\Z^n$ for every $i$. Since $\pi$ is linear we have that
$$\pi(\Z^n)\subset \frac{1}{k}\Z^n.$$
Since $\#\left(\Supp(\xi)\cap\{x\in\R^n \mid  \o\cdot x<-\e\}\right)<\infty$ for any $\e>0$, Lemma \ref{characterization_tau_1} shows that $\o\notin\tau_1(\xi)$. Therefore there is a continuous positive order $\preceq$ refining $\leq_\o$, a $\preceq$-non-negative strongly convex rational cone $\s$ and $\g\in \Z^n$ such that 
$$\Supp(\xi)\subset \g+\s.$$
Because  $\pi$ is a rational cone, $\pi(\s)$ is a rational cone which is $\leq_\o$-non-negative since
$$\forall u\in\R^n\ \ \ \o\cdot u=\o\cdot\pi(u).$$
Since $\pi(\s)$ is rational and contained in $\mathcal L^\perp$, by \eqref{eeqq} we have that
$$\pi(\s)\cap\{x\in\R^n\ \mid\ \o\cdot x=0\}=\{0\}.$$
Therefore for every $\l\in\R$ 
$$\#\left(\left(\pi(\g)+\pi(\s)\right)\cap \frac{1}{k}\Z^n\cap \{x\mid \o\cdot x<\l\}\right)<\infty.$$
In particular 
$$\#\left(\pi(\Supp(\xi))\cap \{x\mid \o\cdot x<\l\}\right)<\infty.$$
Set 
$$\mu:=\max\left\{\o\cdot \b \mid \b\in\pi(\Supp(\xi))\cap\{x\mid \o\cdot x<0\}\right\}.$$
Then
$$\Supp(\xi)\cap\{x\in\R^n \mid \o\cdot x<0\}=\Supp(\xi)\cap\left\{x\in\R^n \mid  \o\cdot x<\frac{\mu}{2}\right\}$$
which is finite by hypothesis.

\end{proof}

The next result is the main theorem of this part. It means that $\t_0(\xi)$ is a kind of maximal dual cone of $\Supp(\xi)$. This result may seem to be a bit technical but it will be very useful in the sequel.

\begin{theorem}\label{main_thm}
	Let $\sigma$ be a strongly convex cone containing the first orthant, let  $\xi\in \K [[\sigma]]_{x_1\cdots x_n}\setminus \K[[x]]_{x_1\cdots x_n}$ be algebraic over $\K[[x]]$. Then ${\R_{>0}}^n\backslash(\t_0(\xi)\cup\t_1(\xi))\neq\emptyset$ and, for every 
	$\omega\in {\R_{>0}}^n\backslash(\t_0(\xi)\cup\t_1(\xi))$, there exists 
	 a Laurent polynomial $p(x)\in \K [ x]_{x_1\cdots x_n}$ such that
	  $$\# \left(\Supp( \In_\omega (\xi +p) )\right)=\infty.$$
\end{theorem}

\begin{proof}
Since $\xi\notin\K[[x]]_{x_1\cdots x_n}$ we have that $\t_0(\xi)\subsetneq{\R_{>0}}^n$ by Corollary \ref{proper_tau_0}. On the other hand since $\xi\in \K [[\sigma]]_{x_1\cdots x_n}$ for every $\o\in \int(\s^\vee)$ we have that $\o\in \t_0(\xi)$ by Lemma \ref{IntDeDualEnTau0}. Thus ${\R_{>0}}^n\backslash \t_0(\xi)$ is closed and different from $\emptyset$ or ${\R_{>0}}^n$. Since ${\R_{>0}}^n$ is connected then ${\R_{>0}}^n\backslash \t_0(\xi)$ is not open so ${\R_{>0}}^n\backslash(\t_0(\xi)\cup\t_1(\xi))\neq\emptyset$.

By definition, for every ${\o}\in{\R_{>0}}^n\backslash ( \tau_0(\xi )\cup\t_1(\xi))$ the following set is non empty and bounded from above by Lemmas \ref{characterization_tau_0} and \ref{characterization_tau_1}:

$$A_{\o}:=\left\{ \l\in\R\mid \# \left(\Supp (\xi )\cap \left\{  u\in\R^n\mid u\cdot{\o}< \l\right\}\right)<\infty\right\}.$$

We fix a vector $\o\in {\R_{>0}}^n\backslash ( \tau_0(\xi )\cup\t_1(\xi))$ and we set 
$$\l_0:=\sup A_{\o}.$$
 
By applying a translation (i.e. by multiplying $\xi$ by a monomial) we may assume that $\l_0=0$.\\
By Lemma \ref{lemma_ref} 
$$\#\left(\Supp(\xi)\cap\{x\in\R^n\ \mid \ \o\cdot x<0\}\right)<\infty.$$
Therefore there exists a Laurent polynomial $p$ such that 
$$\Supp(\xi+p)\subset \{u\in\R^n \mid u\cdot{\o}\geq 0\}.$$
Let us set $\xi':=\xi+p$. Since $\o\notin \t_1(\xi)=\t_1(\xi')$ there exists a continuous positive order $\preceq$ refining $\leq_{\o}$ such that $\xi'={\xi'}_i^\preceq$ for some $i$, where the ${\xi'_i}^\preceq$ are the roots of the minimal polynomial of $\xi'$ in $\Sc$. In particular there exists a rational strongly convex cone $\t$ which is $\preceq$-non-negative and such that $\xi'\in\K[[\t]]_{x_1\cdots x_n}$. By Lemma \ref{initial_rational}, $t=\nu_\o(\xi')$ is well defined. \\
If $t>0$ then $\Supp(\xi')\subset \{u\in\R^n\mid u\cdot\o\geq t\}$ which contradicts the fact that $0=\sup A_{\o}$. Thus $t=0$. If 
$$\#\left(\Supp(\xi')\cap\{u\in\R^n\mid u\cdot\o=0\}\right)<\infty$$
there exists a Laurent polynomial $p'$ such that $\Supp(\xi'+p')\subset \{u\in\R^n\mid u\cdot\o>0\}$.
As before $t':=\nu_\o(\xi'+p')$ is well defined and $t'>0$. Thus 
$$\#\left(\Supp(\xi')\cap\{u\in\R^n\mid u\cdot\o<t'\}\right)<\infty$$
which contradicts again the fact that $0=\sup(A_{\o})$.\\
Hence we have 
$$\Supp(\xi+p)\subset \{u\in\R^n \mid u\cdot{\o}\geq 0\}$$
and
$$\#\left(\Supp(\xi+p)\cap\{u\in\R^n\mid u\cdot\o=0\}\right)=\infty.$$
This proves that
  $$\# \left(\Supp( \In_\omega (\xi +p) )\right)=\infty.$$
\end{proof}

\begin{remark}\label{independent}
If $\o=(\o_1,\ldots,\o_n)\in{\R_{>0}}^n\backslash(\t_0(\xi)\cup\t_1(\xi))$ then $\dim_\Q(\Q\o_1+\cdots+\Q\o_n)\leq n-1$. Indeed if the $\o_i$ were $\Q$-linearly independent any two different monomials $x^\a$ and $x^\b$ would have different valuations: $\nu_\o(x^\a)\neq\nu_\o(x^\b)$. In particular we would not have 
 $$\# \left(\Supp( \In_\omega (\xi +p) )\right)=\infty.$$
\end{remark}

\begin{corollary}\label{line}
	Let $\sigma$ be a strongly convex cone containing the first orthant, let  $\xi\in \K [[\sigma]]_{x_1\cdots x_n}\setminus \K[[x]]_{x_1\cdots x_n}$ be algebraic over $\K[[x]]$. There exists $\omega\in {\R_{> 0}}^n\backslash(\t_0(\xi)\cup\t_1(\xi)$, a Laurent polynomial $p\in \K [ x]_{x_1\cdots x_n}$, a point $\g\in\Z^n$ and a vector $v\in\Z^n$ such that $\# \left(\Supp(\In_\omega (\xi +p) )\right)$ is not finite and 
	$$\Supp ( \In_\omega (\xi+p) )\subset \{ \g + kv \mid k\in \Z_{\geq 0}\}.$$
\end{corollary}

\begin{proof}
Since $\xi\notin \K[[x]]_{x_1\cdots x_n}$ the cone $\t_0(\xi)$ is not equal to ${\R_{>0}}^n$ by Corollary \ref{proper_tau_0}, and it has full dimension by Lemma \ref{DualFuertConvDimMax}  since it contains $\int(\s^\vee)$ (see Lemma \ref{IntDeDualEnTau0}). So let $a=(a_1,\ldots,a_n)$ be a point of the boundary of $\t_0(\xi)$ in ${\R_{> 0}}^n$. 

Since both $\t_0(\xi )$ and $\t_1(\xi )$ are open, $a$ is in ${\R_{> 0}}^n\setminus \t_0(\xi )\cup \t_1(\xi )$. By Theorem \ref{main_thm} there exists 
	 a Laurent polynomial $p(x)\in \K [ x]_{x_1\cdots x_n}$ such that
	  $$\# \left(\Supp( \In_a (\xi +p) )\right)=\infty.$$

We are going to find $\omega\in{\R_{> 0}}^n\setminus \t_0(\xi )\cup \t_1(\xi )$ such that the set  $\Supp( \In_\o (\xi +p) )$ is contained in a line.

Let $B$ be the open ball centered in $a$ of radius $r:=\min\{a_i\}/2>0$. Then $C:=\t_0(\xi)\cap B$ is an open relatively compact convex subset of ${\R_{> 0}}^n$. Let $S^{n-1}_{2r}$ be the  sphere of radius $2r$ centered in $a$.\\
 We claim that the projection $\pi$ of $S^{n-1}_{2r}$ from its center onto the boundary of $C$ is continuous: \\
Indeed let $(u_n)_{n\in\Z_{\geq 0}}$ be a sequence of $S_{2r}^{n-1}$ that converges to $u\in S^{n-1}_{2r}$. Since $C$ is relatively compact its closure is compact and its boundary $\partial C$ is also compact and there exists a subsequence $(\pi(u_{\phi(n)}))_{n\in\Z_{\geq 0}}$ that converges to a vector $v\in\partial C$. Since $u_n\longrightarrow u$, the half-lines $L_{u_n}$ ending at $a$ and passing through $u_n$ converge to the half-line $L_u$ passing $u$ and ending at $a$. Since $\pi(u_{\phi(n)})\in L_{u_n}$ for every $n$ we have that $v\in L_u\cap\partial C$. But $L_u\cap \partial C=\{\pi(u)\}$ since $C$ is convex so $v=\pi(u)$. Thus the only limit point of $(\pi(u_{n}))_{n\in\Z_{\geq 0}}$ is $\pi(u)$ and $\partial C$ being compact the sequence $(\pi(u_n))_{n\in\Z_{\geq 0}}$ converges to $\pi(u)$.

In particular, since the set of lines generated by vectors with $\Q$-linearly independent coordinates is dense in $\P(\R^n)$,  there exists
a half-line  $L$, generated by a vector whose  coordinates are linearly independent over $\Q$, that intersects the boundary of $C$ in a point $\o=(\o_1,\ldots,\o_n)$ such that $\Vert a-\o\Vert <r$. Thus $\o$ will  not be on the boundary of $B$ so it is on the boundary of $\t_0(\xi)$. Since $\t_1(\xi)$ is open and disjoint from $\t_0(\xi)$ then $\o\in{\R_{>0}}^n\backslash(\t_0(\xi)\cup\t_1(\xi))$.\\
Since $L$ is generated by a vector with $\Q$-linearly independent coordinates we have that
\begin{equation}\label{independent2}\dim_\Q(\Q\o_1+\cdots+\Q\o_n)\geq n-1\end{equation}
and this inequality is in fact an equality by Remark \ref{independent}.
By Theorem \ref{main_thm} there exists 
	 a Laurent polynomial $p(x)\in \K [ x]_{x_1\cdots x_n}$ such that
	  $$\# \left(\Supp( \In_\omega (\xi +p) )\right)=\infty.$$
	  But the set of exponents $\a\in\Z^n$ such that
$$\a\cdot\o=\nu_\o(\xi+p)$$ is included in a line by \eqref{independent}. Such line has the form $\g+\R v$ for some $\g\in\Z^n$ and $v\in\Z^n$. If the coordinates of $v$ are assumed to be coprime then
$$(\g+\R v)\cap\Z^n=\g+\Z v.$$
This proves the corollary.

\end{proof}

\begin{remark}\label{TieneCoordenadaPositivayNegativa}
  Notice that, in the proof of Corollary \ref{line}, the vector $v\in\Z^n$ is in $\omega^\perp$ where $\omega\in {\R_{>0}}^n$. Then $v$ has at least one coordinate that is negative and one coordinate that is positive.
\end{remark}






\section{Gap theorem}\label{gap_thmm}

\begin{definition}
An algebraic power series is a power series $f(x)\in\K[[x]]$, such that
$$P(x,f(x))=0$$ for some non-zero polynomial $P(x,y)\in\K[x,y]$ where $y$ is a single indeterminate. The set of algebraic power series is a local subring of $\K[[x]]$ denoted by $\K\langle x\rangle$.\\

\end{definition}


\begin{proposition}\label{thm_alg_power}
Let $\xi=\sum_{\a}\xi_\a x^\a\in\K[[\s]]_{x_1\cdots x_n}\backslash \K[[x]]_{x_1\cdots x_n}$ be algebraic over $\K[[x]]$ where $\s$ is a strongly convex cone containing the first orthant and $\K$ is a characteristic zero field. Then there exist $\g\in \Z^n$, a strongly convex cone $\s'\subset\s$ containing the first orthant, a half-line $L\subset \s'$ and $v\in\Z^n$ such that
\begin{itemize}
\item[i)] $L={\R_{>0}}v$ is a vertex of $\s'$,
\item[ii)] If we set $\xi_{\g+L}=\sum_{\a\in\g+L}\xi_\a x^\a$ then $\xi_{\g+L}=x^\g F(x^v)$ where $F(T)\in\K\langle T\rangle\backslash\K[T]$ and $T$ is a single variable,
\item[iii)] $\xi\in\K[[\s']]_{x_1\cdots x_n}$.
\item[iv)] The vector $v$ has at least one positive and one negative coordinates.
\end{itemize}
\end{proposition}

\begin{proof}
By Corollary \ref{line} there exist $\o\in{\R_{>0}}^n\backslash(\t_0(\xi)\cup\t_1(\xi))$, a Laurent polynomial $p$ and a vector $\g\in\Z^n$ such that $\Supp(\xi+p)\subset \g+\langle\o\rangle^\vee$ and $\Supp(\xi+p)\cap(\g+\langle\o\rangle^\vee)\subset \g+L$ is infinite where $L$ is a half-line ending at the origin. Moreover we may assume that 
$$\dim_{\Q}(\Q\o_1+\cdots+\Q\o_n)=n-1$$
as shown in the proof of Corollary \ref{line} (see \eqref{independent2}). We can also assume that none of the monomials of $p$ lie on $\g+L$. \\
 We have that $\o\notin \t_1(\xi)=\t_1(\xi+p)$ by Lemma \ref{characterization_tau_1}. Thus by Theorem \ref{TeoremaAlgCerrado} the support of $\xi+p$ is included in $\g'+\s'$ where $\g'$ is a vector of $\Z^n$ and $\s'$ is a rational strongly convex cone included in $\langle \o\rangle^\vee$. Since $\dim_{\Q}(\Q\o_1+\cdots+\Q\o_n)=n-1$, $\s'\cap\langle\o\rangle^\vee=L$ and $L$ is a vertex of $\s'$. Thus by Lemma \ref{SiEstaContenidoEnDosDesplazadosDeCono} the support of $\xi$ is included in $\g''+\s'$ for some $\g''\in\Z^n$, hence $\xi\in\K[[\s']]_{x_1\cdots x_n}$. Lemma \ref{SiEstaContenidoEnDosDesplazadosDeCono} allows us to replace $\s'$ by $\s'\cap\s$, thus we may assume that $\s'\subset\s$.\\
 By Lemma \ref{LaUniondeConosTrasladadosEstaEnUnconoTrasladado} we have that $\xi\in\K[[\s']]_{x_1\cdots x_n}$ if and only if $\xi+p\in\K[[\s']]_{x_1\cdots x_n}$. Moreover $\xi$ is algebraic over $\K[[x]]$ if and only if $\xi+p$ is algebraic over $\K[[x]]$. This allows us to replace $\xi$ by $\xi+p$ in the rest of the proof.

Let $v=(v_1,\ldots,v_n)\in\Z^n$ be such that $\R_{\geq 0} v=L$. We may assume that the $v_i$ are globally coprime. Moreover by Remark \ref{TieneCoordenadaPositivayNegativa} we may assume that $v$ has at least one positive and one negative coordinates. In this case we have the following lemma:

\begin{lemma}\label{lemma_ini}
Let $g$ be a non-zero Laurent polynomial (resp. power series) whose support is included in $\g_0+L$ for some $\g_0\in\Z^n$, and let $v=(v_1,\ldots,v_n)\in\Z^n$ be such that $\R_{\geq 0} v=L$ and such that the $v_i$ are globally coprime. Then there exists a one variable polynomial (resp. power series) $G(T)$ such that
$$g(x)=x^{\g_0}G(x^v).$$
\end{lemma}

\begin{proof}[Proof of Lemma \ref{lemma_ini}]
By assumption the support of $g/x^{\g_0}$ is included in $L$. Now if $cx^\a$ is a non-zero monomial of $g/x^{\g_0}$ then $\a=kv$ for some real number $k\geq 0$. Since the $\a_i$ are integers and the $v_i$ are globally coprime then $k\in\Z_{\geq 0}$. Hence we have $$cx^\a=G(x^v)$$ with $G(T)=cT^k$.
\end{proof}

Since $\xi$ is algebraic over $\K[[x]]$ there exist an integer $d$ and formal power series $a_0,\ldots, a_d\in\K[[x]]$ such that
$$a_d\xi^d+\cdots+a_1\xi+a_0=0.$$
Thus 
\begin{equation}\label{rel_alg}\sum_{i\in E} \In_{\o}(a_i)\In_{\o}(\xi)^i=0\end{equation}
where
$$E= \{i\in\{0,\ldots,n\}\ / \ \nu_{\o}(a_i\xi^i)=\min_j\nu_{\o}(a_j\xi^j)\}.$$

By Lemma \ref{lemma_ini} for every $i\in E$ there exist $\g_i\in\Z^n$ and a polynomial $P_i(T)\in\K[T]$
such that 
$$\In_{\o}(a_i)=x^{\g_i}P_i(x^v)$$
and (since $\Supp(\In_{\o}(\xi))\subset \g+ L$) there exists  a power series $F(T)\in\K[[T]]$ such that 
$$\In_{\o}(\xi)=x^{\g}F(x^v).$$
Thus Equation \eqref{rel_alg} yields the relation
\begin{equation}\label{eq_ini} \sum_{i\in E} x^{\g_i+\g} P_i(x^v)F(x^v)^i=0.\end{equation}

But for any monomial $(x^v)^k$ we have $\nu_{\o}((x^v)^k)=k(\o\cdot v)=0$ thus we have 
$$\nu_{\o}(x^{\g_i+\g})=\o\cdot(\g_i+\g)=\o\cdot(\g_j+\g)=\nu_\o(x^{\g_j+\g}) \ \ \forall i,j\in E$$
so
$$\nu_{\o}(x^{\g_i-\g_j})=0 \ \ \forall i,j\in E.$$
Hence for all $i,j\in E$, $\g_i-\g_j\in \R v=L\cup(-L)$. Let $i_0$ be an element of $E$ such that $\g_i-\g_{i_0}\in L$ for all $i\in E$ and for every $i\in E$ let $k_i$ be the integer such that $\g_i-\g_{i_0}=k_iv$. 
Thus Equation \eqref{eq_ini} gives the relation
$$\sum_{i\in E}x^{\g_i-\g_{i_0}}P_i(x^v)F(x^v)^i=\sum_{i\in E}(x^{v})^{k_i}P_i(x^v)F(x^v)^i=0.$$

or 
$$\sum_{i\in E}T^{k_i}P_i(T)F(T)^l=0.$$
This shows that $F(T)$ is an algebraic power series. But $F(T)$ is not a polynomial since $\g+L$ contains an infinite number of monomials of $\xi$ so $F(T)$ is the sum of an infinite number of monomials.
\end{proof}


We can now state the second main result of this work:

\begin{theorem}\label{gap_cor}
Let $\xi\in \K[[\s]]_{x_1\cdots x_n}\backslash \K[[x]]_{x_1\cdots x_n}$ be algebraic over $\K[[x]]$ where $\s$ is a strongly convex cone containing the first orthant and $\K$ is a characteristic zero field. Let $\o=(\o_1,\ldots,\o_n)\in\int(\s^\vee)$. We expand $\xi$ as
$$\xi=\sum_{i\in \Z_{\geq 0}}\xi_{k(i)}$$
where  
\begin{itemize}
\item[i)] for every $k \in \Gamma=\Z\o_1+\cdots+\Z\o_n$, $\xi_k$ is a (finite) sum of monomials of the form $cx^\a$ with $\o\cdot\a=k$,
\item[ii)] the sequence $k(i)$ is a strictly increasing sequence of elements of $\Gamma$,
\item[iii)] for every integer $i$, $\xi_{k(i)}\neq 0$.
\end{itemize}

 Then there exists a constant $C>0$ such that
 $$k(i+1)\leq k(i)+C\ \  \ \ \forall i\in\Z_{\geq 0}.$$
\end{theorem}

\begin{proof}
Let us consider the half-line $L$ and the algebraic power series $F(T)$ of Proposition \ref{thm_alg_power}.
Since $\K$ is a characteristic zero field the algebraic power series $F(T)$ is a $D$-finite power series (see \cite{St} Theorem 2.1). Then again because $\text{char}(\K)=0$, if we set $F(T)=\sum_{m=0}^{\infty} a_mT^m$, there exist an integer $N$ and polynomials $Q_0(m),\ldots,Q_N(m)\in\K[m]$ such that (see \cite{St} Theorem 1.5)
$$Q_0(m)a_m+Q_1(m+1)a_{m+1}+\cdots+Q_N(m+N)a_{m+N}=0\ \ \ \forall m\geq 0.$$
In particular since $F(T)$ is not a polynomial, if we set 
$$r=\max\{|z|\ / Q_i(z)=0 \text{ for some }i\},$$
we cannot have
$$a_{m+1}=\cdots=a_{m+N}=0$$
for some integer $m> r$. Otherwise, by induction on $m$, we see that all the coefficients $a_m$ would be zero for $m>r$. In particular for every integer $m$  the set 
$$\{m,m+1,\ldots,m+2(N+r)\}$$
contains at least two different  elements $i$ and $j$ such that $a_i\neq 0$ and $a_j\neq 0$.

We have $\xi_{\g+L}=x^\g\sum_{m}a_m(x^v)^m$ and $\nu_\o(x^\g(a_m(x^v)^m))=\o\cdot\g+m\o\cdot v$ for every $m$ with $a_m\neq 0$. Then we have that
$$k(i+1)-k(i)\leq 2(N+r)\o\cdot v\ \ \ \ \forall i\in\Z_{\geq 0}.$$
This proves the theorem with $C:=2(N+r)\o\cdot v$.
\end{proof}


\begin{remark}\label{K[x]}
Let $\xi\in \K[[\s]]_{x_1\cdots x_n}\backslash \K[[x]]_{x_1\cdots x_n}$ be algebraic over $\K[x]$ where $\s$ is a strongly convex cone containing the first orthant and $\K$ is a characteristic zero field. One may prove more easily Theorem \ref{gap_cor} in this case. \\
Indeed one may find a bijective  linear map $\phi : \R^n\to\R^n$ with integer coefficients such that $\phi(\s)\subset {\R_{\geq 0}}^n$. Then one can show that this  map induces a monomial morphism $\tilde\phi : \K[[\s]]_{x_1\cdots x_n}\to\K[[x]]_{x_1\cdots x_n}$. Thus $\tilde\xi:=\tilde \phi(\xi)\in\K[[x]]_{x_1\cdots x_n}$ is algebraic over $\K[x]$ and it is not very difficult to check that the conclusion of Theorem \ref{gap_cor} is satisfied by $\xi$ if and only if it is satisfied by $\tilde\xi$. But here, since $\tilde\xi$ is algebraic over $\K[x]$, $\tilde \xi$ is D-finite (see \cite{Li}) and it is not too difficult to prove that the conclusion of Theorem \ref{gap_cor} is satisfied by a D-finite power series.
\end{remark}

\begin{example}
Let  $\xi:=\sum_{i=0}^{\infty}\left(\frac{x_2}{x_1}\right)^{i^2}$ and let us choose $\o_1=1$ and $\o_2=2$. Using the notations of Theorem \ref{gap_cor} we have that 
$$\xi=\sum_{i\in \Z_{\geq 0}}\xi_{k(i)}$$
 where $k(i)=i^2$ and $\xi_{i^2}=\left(\frac{x_2}{x_1}\right)^{i^2}$. Thus $\xi$ is not algebraic by Theorem \ref{gap_cor}.\\
%
\end{example}

 
\section{Diophantine approximation for Laurent series}
Theorem \ref{gap_cor} have some implications in term of diophantine approximation. Before giving these implications we need to introduce some background. \\
Every vector $\o\in\R_{>0}^n$ defines a monomial valuation on $\K(\!(x)\!)$ as follows:
 for a nonzero formal power series $f$ written as $f=\sum_{\a\in\Z_{\geq 0}^n}f_\a x^\a$ where $f_\a\in\K$ for every $\a$ we set
 $$\nu_\o(f)=\min\{\o\cdot\a \mid f_\a\neq 0\}$$
 and for any nonzero formal power series $f$ and $g$ we set
$$\nu_\o\left(\frac{f}{g}\right)=\nu_\o(f)-\nu_\o(g).$$
This valuation defines a non-archimedean norm, which makes $\K\( x\)$ a topological field, as follows:
$$\left| \frac{f}{g}\right|_\o=e^{-(\nu_\o(f)-\nu_\o(g))}$$
for every nonzero power series $f$ and $g$. But for $n\geq 2$ this topological field is not a complete field. It is possible to describe quite easily the completion of $\K\(x\)$ for such a topology. Its elements are the Laurent series $\sum_kf_k$ such that $f_k\in\K\(x\)$ for every integer $k$ and $\nu_\o(f_k)\longrightarrow +\infty$ when $k$ goes to infinity. In particular  a Laurent power series with support in $\g+\s$ for some strongly convex cone will be an element of the completion of $\K\(x\)$ for the topology induced by $\nu_\o$ as soon as $\o\cdot u>0$ for every $u\in\s\backslash\{0\}$. Then the following  theorem on Diophantine approximation provides a condition on the algebraicity of such an element of the completion:

\begin{theorem}\label{thm_diop_anc}\cite{Ro}\cite{II}\cite{Ro2}
Let $\K$ be a field of any characteristic. 
Let $\xi$ be in the completion of $\K\(x\)$ for the topology induced by $\nu_\o$. If $\xi\notin \K\(x\)$ is algebraic over $\K\(x\)$ then there exist two constants $C>0$ and $a\geq 1$ such that
\begin{equation}\label{diop_inequality}\left|\xi-\frac{f}{g}\right|_\o\geq C|g|_\o^a\ \ \ \forall f,g\in \K[[x]].\end{equation}
\end{theorem}

We can rewrite Inequality \eqref{diop_inequality} using the valuation $\nu_\o$ as follows: there exist two constants $a\geq 1$ and $b\geq 0$ (here $C=e^{-b}$) such that 
$$\nu_\o\left(\xi-\frac{f}{g}\right)\leq a\nu_\o(g)+b\ \ \ \forall f,g\in\K[[x]].$$

This result allows to show that some Laurent series with support in a strongly convex cone are not algebraic over $\K\(x\)$ as in the following example:

 \begin{example}\label{ex_tr}
Here $n=2$.
We set 
$$\s:=(-1,1)\R_{\geq 0}+(1,0)\R_{\geq 0}\subset \R^2.$$ This is a strongly convex cone of $\R^2$. Let  us consider the following Laurent series with support in $\s$:
$$\xi:=\sum_{i=0}^{\infty}\left(\frac{x_2}{x_1}\right)^{i!}.$$
Then  $\xi$ is an element of the completion of $\K\(x\)$ for the topology induced by $\nu_\o$ for any $\o=(\o_1,\o_2)\in{\R_{>0}}^2$ such that $\o_2>\o_1$. Moreover
$$\nu_{\o}\left(\xi-\sum_{i=0}^N\left(\frac{x_2}{x_1}\right)^{i!}\right)=(N+1)!(\o_2-\o_1)=\frac{\o_2-\o_1}{\o_1}(N+1) \nu_{\o}(x_1^{N!})\ \ \ \ \forall N\in\Z_{\geq 0}.$$ Thus there do not exist  constants $a$ and $b$ such that 
$$a\nu_{\o}(x_1^{N!})+b\geq \nu_{\o}\left(\xi-\sum_{i=0}^N\left(\frac{x_2}{x_1}\right)^{i!}\right)\ \ \forall N\in\Z_{\geq 0}.$$
Hence $\xi$ is not algebraic over $\K\(x\)$ by Theorem \ref{thm_diop_anc}.
\end{example}
Here the key argument is that Theorem \ref{thm_diop_anc} implies that the ratio of the valuations of two nonzero consecutive terms in the expansion  of a given algebraic Laurent series is bounded. So we remark that Theorem \ref{gap_cor} provides a stronger criterion of algebraicity in characteristic zero.

We can give an analogy with the problem of transcendence of real numbers. Here the analogue of $\K[[x]]$ (resp. $\K\(x\)$, resp. its completion for the topology induced by $\nu_\o$) is $\Z$ (resp. $\Q$, resp. the completion of $\Q$ for the usual topology induced by the absolute value, i.e. $\R$). The analogue of Theorem \ref{thm_diop_anc} is the classical Liouville's  Theorem which allows to prove exactly as done in Example \ref{ex_tr} that a real number as $\sum_{i=0}^{\infty}\frac{1}{2^{i!}}$ is transcendental (see \cite{La} for Liouville's Theorem).  \\
Let us stress the fact that this analogy has some limits. Indeed in Theorem \ref{thm_diop_anc} one cannot take $a=[\K\(x\)[\xi]:\K\(x\)]$ as in Liouville's  Theorem as seen in \cite{Ro}, while Roth's Theorem allows us to replace $a$ by any constant $a'>2$ in the classical case. In fact the main difference between the two situations is that $\Z^*$ is in the complement of the unit open ball of $\R$ while $\K[[x]]$ is included in a unit open ball in the completion of $\K\(x\)$. In particular the proof of Theorem \ref{thm_diop_anc} is completely different from the proof of Liouville's  Theorem. Indeed the key fact of the proof of Liouville's Theorem is that $\Z^*$ is in the complement of the unit open ball of $\R$. The fact that the Diophantine approximation Liouville's Theorem holds in the setting of power series in spite of this main difference is a striking fact and a strong motivation for a deeper investigation of this mysterious analogy.  
\\
\begin{remark}
One can push further this analogy and remark that  it is well known that Liouville's  Theorem cannot be used to prove that some numbers as $\sum_{i=0}^{\infty}\frac{1}{2^{i^2}}$ are transcendental since they are not well enough approximated by rationals. In fact the transcendence of this real number has been proven relatively recently by Y. Nesterenko (see \cite{Ne1} or \cite{Ne2}) while it has been an open problem for more than one century.  In fact the original Nesterenko theorem concerns the algebraic independence of 4 modular functions when they are evaluated at points in the open unit disc. One corollary of this fact is that the series $\sum_i \frac{1}{q^{i^2}}$ is transcendental over the rational numbers for every algebraic complex number $q$ in the open unit disc. This is an example of a transcendental real number whose sequence of truncations in basis $q$ for some integer $q$ does not converge too quickly.\\
So Theorem \ref{gap_cor} can be seen as a kind of non-archimedean analogue of Nesterenko's Theorem. In fact it is more general than this Nesterenko's Theorem since it applies to any $\xi$ whose sequence of truncations does not converge too quickly to $\xi$. \\
\end{remark}
Let us remark that in Theorem \ref{thm_diop_anc} the inequality remains true if we replace $a$ by a greater constant. But the smaller is $a$, worse $\xi$ is approximated by elements of $\K\(x\)$.\\
Using Theorem \ref{gap_cor} we can prove the following Diophantine approximation type result asserting that for a given algebraic element $\xi$, there exists a well chosen norm $|\cdot|_{\nu_\o}$, such that the constant $a$ of Theorem \ref{thm_diop_anc} can be chosen equal to 1 (in the case where $\xi$ is approximated by elements of the form $\frac{g}{x^\b}$ with $g\in\K[[x]]$ and $\b\in{\Z_{\geq 0}}^n$) with this norm. This means that $\xi$ is very badly approximated by elements of the form $\frac{g}{x^\b}$ for this particular norm.
%

\begin{theorem}\label{diop_cor}
Let $\xi\in\K[[\s]]_{x_1\cdots x_n}\backslash \K[[x]]_{x_1\cdots x_n}$ be algebraic over $\K[[x]]$ where $\s$ is a strongly convex cone containing the first orthant and $\K$ is a characteristic zero field. Then there exist a strongly convex cone $\s'\subset \s$ containing the first orthant with $\xi\in\K[[\s']]_{x_1\cdots x_n}$, $\o\in\int({\s'}^\vee)$ and a constant $b\in\R$ such that for all $g\in\K[[x]]$ and all $\b\in{\Z_{\geq0}}^n$ we have that
\begin{equation}\label{diop}\nu_{\o}\left(\xi-\frac{g}{x^\b}\right)\leq  \o\cdot \b+b\end{equation}
or equivalently
\begin{equation}\label{diop'}\left|\xi-\frac{g}{x^\b}\right|_{\nu_{\o}}\geq  C|x^\b|_{\nu_{\o}}\end{equation}
where $C=e^{-b}$.
\end{theorem}

\begin{proof}

 Let $\xi_1$, $\xi_2\in\K[[\s]]_{x_1\cdots x_n}$ be such that $\Supp(\xi_1)\cap\Supp(\xi_2)=\emptyset$.
Then we have for every $\b\in{\Z_{\geq0}}^n$
\begin{equation}\label{sum_diop}\sup_{g\in\K[[x]]}\nu_\o\left(\xi_1+\xi_2-\frac{g}{x^\b}\right)= \min\left\{\sup_{g_1\in\K[[x]]}\nu_\o\left(\xi_1-\frac{g_1}{x^\b}\right),\sup_{g_2\in\K[[x]]}\nu_\o\left(\xi_2-\frac{g_2}{x^\b}\right)\right\}.\end{equation}

Thus if  there exists a constant $b$  such that for all $g_1\in\K[[x]]$ and all $\b\in{\Z_{\geq0}}^n$ we have
$$\nu_\o\left(\xi_1-\frac{g_1}{x^\b}\right)\leq \o\cdot\b+b$$
then 
$$\nu_\o\left(\xi_1+\xi_2-\frac{g}{x^\b}\right)\leq \o\cdot\b+b\ \ \ \forall g\in\K[[x]],\ \forall \b\in{\Z_{\geq0}}^n.$$
 
Hence in order to prove Inequality \eqref{diop} we are allowed to replace $\xi$ by a power series $\xi_1\in\K[[\s]]_{x_1\cdots x_n}$ such that there exists $\xi_2\in\K[[\s]]_{x_1\cdots x_n}$ with $\Supp(\xi_1)\cap\Supp(\xi_2)=\emptyset$ and $\xi_1+\xi_2=\xi$. In particular, using the notations of Proposition \ref{thm_alg_power}, since $\xi=\xi_{\g+L}+\xi_2$  where $\Supp(\xi_2)\cap (\g+L)=\emptyset$, we can assume that $\xi=\xi_{\g+L}$. Moreover 
$\s'$ will denote the cone given in Proposition \ref{thm_alg_power}, so $L$ is a vertex of $\s'$. \\

So from now on we assume that $\Supp(\xi)\subset \g+L$. Thus, by Lemma  \ref{lemma_ini} we can write $\xi=x^{\g}F(x^v)$ for some algebraic power series $F(T)$ that is not a polynomial  and, by Remark \ref{TieneCoordenadaPositivayNegativa}, $v\in\Z^n$ is a direction vector of $L$ whose coordinates are coprime and it has at least one positive and one negative coordinate. \\

Let us write
$$F(T)=\sum_{m\in\Z_{\geq 0}}a_mT^m.$$ 
As discussed in the proof of Theorem \ref{gap_cor}, there exist natural numbers $C$ and $M$ such that, for every $m>M$, we have that
\begin{equation}\label{eq-D-finite}
\left\{ i\mid a_i\neq 0 , i\in\{ m,m+1,\ldots ,m +C\}\right\}\neq \emptyset.
\end{equation}

Take $\b\in{\Z_{\geq0}}^n$ and let $\o\in{\Z_{>0}}^n$ be such that $\o\cdot v>0$. Then 
\begin{equation}\label{diop3}\sup_{g\in\K[[x]]}\nu_\o\left(\xi x^\b-g\right)=\sup_{g\in\K[[x]]}\nu_\o\left(\sum_ma_mx^{mv+\b+\g}-g\right)=\nu_\o(a_{m_0}x^{m_0v+\b+\g})\end{equation}
where 
$$m_0=\min\{m\in\Z_{\geq 0}\ / \  a_m\neq 0\text{ and } mv+\b+\g\notin{\Z_{\geq0}}^n\}.$$
By permuting the coordinates of $v$, denoted by $v_1,\ldots,v_n$, we may assume that there exists an integer $k$ such that
 $v_i<0$ for $i\leq k$, $v_i\geq 0$ for $i>k$ and $v_n>0$. 
 
In this case  $mv+\b+\g\notin{\Z_{\geq0}}^n$ if and only if  at least one of the following three situations arises:
\begin{enumerate}
\item[(i)] for some $i>k$, $v_i>0$ and $mv_i+\b_i+\g_i<0$,
\item[(ii)] for some $i>k$, $v_i=0$ and $\b_i+\g_i<0$,
\item[(iii)] for some $i\leq k$,  $mv_i+\b_i+\g_{i}<0$.
\end{enumerate}
In Case (i) we have that $m<-\frac{\b_i+\g_i}{v_i}\leq -\frac{\g_i}{v_i}$, since $\b_i\geq 0$, and $\b_i<-\g_i$. In Case (ii) we have that $\b_i<-\g_i$.

 Case (iii) is equivalent to
$$m>\min_{1\leq i\leq k}\left\{\frac{\b_i+\g_{i}}{-v_i}\right\}.$$ 
Thus we have that
$$m_0\leq \max\{M_1,M_2,M_3\}\leq M_1+M_2+M_3$$
where $M_1=\max\left\{-\frac{\g_i}{v_i}\mid v_i>0 \text{ and } \g_i<0\right\}$, $M_2=\ord_T(F(T))$ and
$$M_3=\min\left\{m\in\Z_{\geq0}\ / \  a_m\neq 0\text{ and }m\geq\left\lfloor \min_{1\leq i\leq k}\left\{\frac{\b_i+\g_{i}}{-v_i}\right\}\right\rfloor +1\right\}$$

$$\leq \min\left\{m\in\Z_{\geq0}\ / \  a_m\neq 0\text{ and }m\geq\left\lfloor \min_{1\leq i\leq k\mid \b_i+\g_i>0}\left\{\frac{\b_i+\g_{i}}{-v_i}\right\}\right\rfloor +1\right\}$$
where $\lfloor c\rfloor$ denotes the integer part of a real number $c$.
Exactly as shown in the proof of Theorem \ref{gap_cor}, by \eqref{eq-D-finite} we have that
$$M_3\leq \lfloor A(\b)\rfloor +1+2C+M$$
where $$A(\b):=\min_{1\leq i\leq k\mid 
\b_i+\g_i>0}\left\{\frac{\b_i+\g_{i}}{-v_i}\right\}.$$

Then by \eqref{diop3} we have that
 $$\sup_{g\in\K[[x]]}\nu_\o\left(\xi x^\b-g\right)\leq (\lfloor A(\b)\rfloor +1+2C+M+M_1+M_2)\o\cdot v+\o\cdot(\b+\g).$$
 Let $i_0\leq k$ be such that $\displaystyle\min_{1\leq i\leq k\mid\b_i+\g_i>0}\left\{\frac{\b_i+\g_{i}}{-v_i}\right\}=\frac{\b_{i_0}+\g_{i_0}}{-v_{i_0}}$.  From now on  let us choose $\o\in\int(\s'^\vee)$ satisfying the conclusion of Lemma \ref{inequalities} given below (with $\t=\s'$). In particular we have that
 $$\lfloor A(\b)\rfloor \o\cdot v\leq \frac{\b_{i_0}+\g_{i_0}}{-v_{i_0}}\o\cdot v \leq \o_{i_0}(\b_{i_0}+\g_{i_0})$$
 since $\b_{i_0}+\g_{i_0}>0$.
 But since $\o\in \s'^\vee$ and ${\R_{\geq 0}}^n\subset \s'$ then $\o_i\geq 0$ for all $i$. Thus we have that
 $$\o_{i_0}(\b_{i_0}+\g_{i_0})\leq \o\cdot\b+\o_{i_0}\g_{i_0}.$$
 Hence we obtain
 $$\sup_{g\in\K[[x]]}\nu_\o\left(\xi x^\b-g\right)\leq 2\o\cdot\b+(1+2C+M+M_1+M_2)\o\cdot v+\o\cdot\g+\o_{i_0}\g_{i_0}.$$
 or
 $$\sup_{g\in\K[[x]]}\nu_\o\left(\xi-\frac{g}{x^\b}\right)\leq \o\cdot\b+(1+2C+M+M_1+M_2)\o\cdot v+\o\cdot\g+\o_{i_0}\g_{i_0}.$$
 This proves the corollary with $b:=(1+2C+M+M_1+M_2)\o\cdot v+\o\cdot\g+\o_{i_0}\g_{i_0}$.

\end{proof}

\begin{lemma}\label{inequalities}
Let $\t$ be a strongly convex rational  cone containing ${\R_{\geq 0}}^n$ and let $v\in\t$ be a vertex of $\t$. Let us assume that there exists $k<n$ with $v_i<0$ for $i\leq k$ and $v_i\geq 0$ for $i>k$.
Then there exists $\o\in\t^{\vee}$ such that
$$0< \o\cdot v< -\o_jv_j\ \ \ \ \forall j\leq k.$$
\end{lemma}

\begin{proof}
For $1\leq i\leq k$ let us denote
$$v^{(i)}:=(v_1,\ldots,v_{i-1},2v_i,v_{i+1},\ldots,v_n)$$
and let $\t'$ be the cone generated by $v^{(1)},\ldots,v^{(k)}$. We claim that
$$\t\cap\t'=\{0\}.$$
Indeed let $\nu\in \t\cap\t'$. Then we can write
$$\nu=\sum_{i=1}^k\l_iv^{(i)}$$
for some $\l_i\geq 0$. Thus we easily check that
$$\nu=(\l_1v_1,\l_2v_2,\ldots,\l_kv_k,0\ldots,0)+(\sum_{i=1}^k\l_i)v.$$
If $\sum_{i=1}^k\l_i\neq 0$ this implies that 
$$v=\frac{1}{\sum_{i=1}^k\l_i}\nu+\frac{1}{\sum_{i=1}^k\l_i}(-\l_1v_1,-\l_2v_2,\cdots,-\l_kv_k,0,\ldots,0)\in\t$$
is not a vertex of $\t$
since $\nu\in\t$ and the $v_i$ are negative for $1\leq i\leq k$. Thus $\sum_{i=1}^k\l_i=0$ and $\l_i=0$ for all $1\leq i\leq k$. This proves that $\nu=0$.\\

Let $\o\in\R^n$ such that $\o\in \Int (-\t'^\vee)$ and $\o\in \Int (\t^\vee)$, i.e. the hyperplane defined by $\o$ separates $\t$ and $\t'$. Then $\o\cdot\nu>0$ for all $\nu\in \t\backslash\{0\}$ and $\o\cdot v^{(i)}<0$ for all $1\leq i\leq k$. This proves the lemma.

\end{proof}

\begin{example}
Set $P(T)=x_1^2T^2-(x_1^2+x_2^2)$. If the $a_k\in\Q^*$ are the coefficients of the following Taylor expansion:
$$\sqrt{1+U}=1+a_1U+a_2U+\cdots+a_kU^k+\cdots$$
then the element
$$\xi:=1+a_1\frac{x_1}{x_2}+a_2\frac{x_1^2}{x_2^2}+\cdots+a_k\frac{x_1^k}{x_2^k}+\cdots\in\K[[\s]]$$
is a root of $P(T)$ where $\s$ is the cone generated by $(1,0)$, $(0,1)$ and $(1,-1)$. Let $\o\in\s^\vee$. Let $N\in\Z_{> 0}$. Then we have
$$\nu_\o\left(\xi-\left(1+\sum_{k=1}^Na_k\frac{x_1^k}{x_2^k}\right)\right)=(N+1)(\o_1-\o_2).$$
On the other hand we can write
$$1+\sum_{k=1}^Na_k\frac{x_1^k}{x_2^k}=\frac{f_N}{x_2^N}$$
for some polynomial $f_N\in\K[x_1,x_2]$. Thus
$$\nu_\o\left(\xi-\left(1+\sum_{k=1}^Na_k\frac{x_1^k}{x_2^k}\right)\right)\bigg/\nu_\o(x_2^N)=\frac{N+1}{N}\frac{\o_1-\o_2}{\o_2}\xrightarrow[N\to \infty]{} \frac{\o_1-\o_2}{\o_2}$$
and this last term is equal to 1 if and only if $\o_1=2\o_2$. This shows that there exists only one $\o\in\s^\vee$ (up to multiplication by a positive scalar) satisfying Theorem \ref{diop_cor}.\\

Now let us pick such a $\o$, for instance $\o=(2,1)$. We have that $\xi\in \K[[\t]]$ for any strongly convex cone $\t$ containing $\s$. For instance we can consider the cone $\t$ generated by $\s$ and $(-2,3)$. It is straightforward to check that $\t$ is strongly convex but $\o\notin \t^\vee$ since $\o\cdot(-2,3)=-1$. This shows that, in Theorem \ref{diop_cor}, we may need to replace $\s$ by a smaller cone $\s'$.

\end{example}

\begin{example}
Let $\K$ be a field of  characteristic $p>0$. We define
$$\xi:=\sum_{k\in\Z_{\geq 0}}x_1^{p^k}x_2^{-p^k}.$$
This is a power series whose support is in $\s=\{(i,j)\in\R^2\ / \  \ i\geq 0, i+j\geq 0\}$. 
This power series is obviously a root of
$$P(T):=T^p-T+\frac{x_1}{x_2}.$$
Thus $\xi\in \K[[\s]]_{x_1\cdots x_n}\backslash\K[[x]]_{x_1\cdots x_n}$ is algebraic over $\K[[x]]$ but it is straightforward to check that Theorem \ref{gap_cor} is not satisfied.\\

On the other hand for $\o=(\o_1,\o_2)\in \int (\s^\vee)$ (i.e $\o$ satisfies $\o_1>\o_2>0$) and $(\b_1,\b_2)\in{\Z_{\geq 0}}^2$ we have
$$\sup_{g\in\K[[x]]}\nu_{\o}\left(\xi x_1^{\b_1}x_2^{\b_2}-g\right)=\o\cdot(p^{k_0}+\b_1,-p^{k_0}+\b_2)$$
where $k_0=\min\{k\in\Z_{\geq 0}\ / \ -p^k+\b_2<0\}$. In particular we have that
$$-p^{k_0-1}+\b_2\geq 0.$$
Thus we have 
$$\sup_{g\in\K[[x]]}\nu_{\o}\left(\xi x_1^{\b_1}x_2^{\b_2}-g\right)=\o\cdot(p^{k_0},-p^{k_0})+\o\cdot\b = (\o_1-\o_2)p^k+\o\cdot\b$$
$$\leq (\o_1-\o_2)p\b_2+\o\cdot\b.$$

Hence for $\o=(p+1,p)$ we see that
$$\sup_{g\in\K[[x]]}\nu_{\o}\left(\xi-\frac{g}{x_1^{\b_1}x_2^{\b_2}}\right)\leq p\b_2=\o_2\b_2\leq \o\cdot\b.$$
This proves that Theorem  \ref{diop_cor} remains valid for $\xi$. So it is a natural open question to know if Theorem \ref{diop_cor} remains valid in general in positive characteristic.
\end{example}

\end{document}